\documentclass[12pt]{amsart}
\usepackage[english]{babel}
\usepackage{amsmath}
\usepackage{amsthm}
\usepackage{amssymb}
\usepackage{mathrsfs}
\usepackage{mathdots}
\usepackage{enumerate,MnSymbol}
\usepackage[notcite, final, notref]{showkeys}
\usepackage{dsfont}
\usepackage{color}
\usepackage{shadow}
\newtheorem{theorem}{Theorem}[section]
\newtheorem{problem}[theorem]{Problem}
\newtheorem{lemma}[theorem]{Lemma}
\newtheorem{proposition}[theorem]{Proposition}
\newtheorem{corollary}[theorem]{Corollary}
\newtheorem{definition}[theorem]{Definition}
\newtheorem{hypothesis}[theorem]{Hypothesis}
\newtheorem{notation}[theorem]{Notation}
\theoremstyle{definition}
\newtheorem{remark}[theorem]{Remark}

\newtheorem{example}[theorem]{Example}
\newcommand{\rkhs}{reproducing kernel Hilbert space }

\newcommand{\A}{\alpha}

\newcommand{\e}{\epsilon}

\newcommand{\M}{\mathsf M}

\title[Beurling-Lax type theorems and Cuntz relations]{Beurling-Lax type theorems and Cuntz relations}

\author[D. Alpay]{Daniel Alpay}
\address{(DA)
Faculty of Mathematics, Physics, and Computation\\
Schmid College of Science and Technology\\
Chapman University\\
One University Drive
Orange, California 92866\\
USA}
\email{alpay@chapman.edu}

\author[F. Colombo]{Fabrizio Colombo}
\address{(FC) Politecnico di
Milano\\Dipartimento di Matematica\\Via E. Bonardi, 9\\20133 Milano\\Italy}
\email{fabrizio.colombo@polimi.it}

\author[I. Sabadini]{Irene Sabadini}
\address{(IS) Politecnico di
Milano\\Dipartimento di Matematica\\Via E. Bonardi, 9\\20133 Milano\\Italy}
\email{irene.sabadini@polimi.it}

\author[B. Schneider]{Baruch Schneider}
\address{(BS) University of Ostrava\\ Department of Mathematics\\
30.dubna 22, 70200 Ostrava\\Czech Republic }
\email{baruch.schneider@osu.cz}

\begin{document}
\maketitle
\begin{abstract}
We prove various Beurling-Lax type theorems, when the classical backward-shift operator is replaced by a general resolvent operator associated with a rational
function. We also study connections to the Cuntz relations. An important tool is a new representation result for analytic functions, in terms of composition and multiplication
operators associated with a given rational function. Applications to the theory of de Branges-Rovnyak spaces, also in the indefinite metric setting, are given.
\end{abstract}


\date{today}
\tableofcontents
\section{Introduction}
\setcounter{equation}{0}
In this paper we discuss decompositions of spaces associated with positive definite kernels and with kernels having a finite number of negative squares,
based on a new representation theorem for analytic functions. To set the framework and motivate our results, let us
consider a $\mathbb C^{p\times p}$-valued function $K(z,w)$ positive definite on a set $\Omega\subset\mathbb C$; let $r$ be a map from a set $\Omega_1\subset\mathbb C$ into $\Omega$, and let $A(z)$ be a
$\mathbb C^{1\times p}$-valued function defined on $\Omega_1$. The function $A(z)K(r(z),r(w))A(w)^*$ is positive definite on $\Omega_1$; elements of the reproducing kernel Hilbert space with reproducing
kernel $A(z)K(r(z),r(w))A(w)^*$ are of the form
\begin{equation}
  \label{fF}
f(z)=A(z)F(r(z)),
\end{equation}
where $F$ varies in the reproducing kernel Hilbert space with reproducing kernel $K(z,w)$.
Depending on various assumptions, one can insure that the above decomposition is unique, and one can then relate the norms of $f$ and $F$. In the present paper we assume that
$r$ is a rational function and that the entries of $A(z)$ span the backward-shift invariant space (also know as state space) $\mathfrak L(r)$ generated by $r$ (see Definition \ref{Statespace}
for the latter). We focus in particular on de Branges and de Branges-Rovnyak spaces  (see \cite{dbbook,dbr1,dbr2,dmk,MR3497010,MR3617311}), in view of
their important role in operator theory, linear system theory and
related topics; see e.g. \cite{MR1638044,Dym_CBMS}.
The properties of these spaces are related to properties of the classical backward-shift operators
\begin{equation}
  \label{totoche}
  (R_\alpha F)(z)=\begin{cases}\,\,\dfrac{F(z)-F(\alpha)}{z-\alpha},\quad z\not=\alpha,\\
    \,\, F^\prime(\alpha), \hspace{1.72cm} z=\alpha,
    \end{cases}
  \end{equation}
  where $F$ is assumed analytic in a neighborhood of $\alpha$. Moreover, various structure theorems characterizing de Branges-Rovnyak spaces in terms of the backward-shift operators
  can be seen as  Beurling-Lax type theorems. For instance, given a signature matrix $J\in\mathbb C^{s\times s}$ (i.e. $J$ satisfies $J=J^*=J^{-1}$) and given a reproducing
  kernel Hilbert space $\mathfrak H$ of
  $\mathbb C^s$-valued functions analytic in an open subset of the complex plane, symmetric with respect to the real line, the equality
\begin{equation}
  \left[ R_\alpha F,G\right] -\left[ F, R_\beta G\right]  -(\alpha-\overline{\beta})\left[ R_\alpha F,R_\beta G\right]=iG(\beta)^*JF(\alpha)
  \label{wdc}
\end{equation}
holds for all $\alpha,\beta\in \Omega$ and all $F,G\in\mathfrak H$ if and only if the reproducing kernel of $\mathfrak H$ is of the form
\begin{equation}
  K(z,w)=\frac{J-\Theta(z)J\Theta(w)^*}{-i(z-\overline{w})},
\end{equation}
see \cite{ball-contrac,dbhsaf1,HM}. The links with operator models stem from the following: When the operators $R_\alpha$ are bounded and with zero kernel, a result of Stone
(see \cite[Theorem 4.10, p. 137]{Stone}) asserts that there is a closed  linear operator $\mathsf A$ such that
\begin{equation}
  \label{model123321}
((\mathsf A-\alpha I)^{-1}f)(z)=zf(z)+c_f,\quad \alpha\in\Omega,
\end{equation}
where $c_f\in\mathbb C$. Properties of $\mathsf A$ and of the space $\mathfrak H$ are closely related.\\

de Branges-Rovnyak spaces come in all shapes and sizes, depending on the application in mind, and the domain of analyticity where they are defined. We here focus on three types of spaces,
usually denoted by the symbols $\mathfrak P(\Theta)$, $\mathfrak P(E_+,E_-)$ and $\mathfrak L(\mathsf N)$, associated respectively with closed to self-adjoint
(or close to unitary), Hermitian (or isometric) and
self-adjoint (or unitary) operators.\\

Although different in aim and approach, the present paper has connections with the works of Kailah and Lev-Ari \cite{lak}, Nudelman \cite{MR1246586}, and the approach developed in the papers
\cite{abds2-jfa,ad-jfa,ad9}. There, one is given two functions analytic in a connected open set $\Omega$, such that the sets
\[
  \begin{split}
    \Omega_+&=\left\{z\in\Omega\,;\, |b(z)|<|a(z)|\right\}\\
    \Omega_-&=\left\{z\in\Omega\,;\, |b(z)|>|a(z)|\right\}
  \end{split}
\]
are not empty. Then the set
\[
\Omega_0=\left\{z\in\Omega\,;\, |b(z)|=|a(z)|\right\}
\]
is also not empty (see e.g. \cite[Exercise 6.1.9, p. 295]{MR3560222}), and the function
\[
k(z,w)=\frac{1}{a(z)\overline{a(w)}-b(z)\overline{b(w)}}
  \]
is positive definite in $\Omega_+$. The reproducing kernels studied are then of the form
\begin{equation}
  \label{sapporo}
\frac{J-\Theta(z)J\Theta(w)^*}{a(z)\overline{a(w)}-b(z)\overline{b(w)}}
\end{equation}
when $z,w$ vary in some open subset of $\Omega_+$, where $J$ is a signature matrix. When such a kernel has a finite number of negative squares, the function $\Theta(z)$ in \eqref{sapporo} can be
written as $M(\sigma(z))$, where $\sigma(z)=\frac{b(z)}{a(z)}$, and where $M(z)$ is analytic in some open subset of the open unit disk.
Our rational function $r$ (or $1-r/1+r$) plays the role of $\sigma$.\\

de Branges-Rovnyak spaces have been furthermore generalized to numerous settings, ranging from the polydisk
(and the theory of Schur-Agler classes) to non-commutative settings and real compact Riemann surfaces setting, and here we take a different avenue and
study the counterparts of de Branges-Rovnyak spaces  when the classical resolvent operator (see \eqref{totoche} below) is replaced by a general resolvent operator associated with a
rational function $r$. Motivated by the Riemann surface case (see \cite[\S 4]{av3}), a generalization of \eqref{totoche}
has been introduced in \cite{AJLV15}. More precisely (in \cite{AJLV15} only the polynomial case was considered) we will assume, unless otherwise stated, that $r$ satisfies:

\begin{hypothesis}
 We will assume that the rational function $r(z)=\frac{p(z)}{q(z)}$,  where $p$ and $q$ are coprime polynomials  satisfying ${\rm deg }\,p\,\ge\, {\rm deg}\, q$.
When this hypothesis is not in force we can replace $r$ by $1/r$.
\end{hypothesis}

We set $N={\rm deg}\, p$, and define
\begin{equation}
\Omega(r)=\left\{\alpha\in\mathbb C\,\mbox{\text{\,such that the equation $r(z)=\alpha$ has $N$ pairwise different roots}}\right\}.
\label{rtyuiop}
\end{equation}
We do not write explicitly the dependence of the roots, say $w_1(\A),\ldots, w_N(\A)$,  on $r$ and $\alpha$, but when it makes the reading easier (such as in the statement and proof of
Lemma \ref{lemm26}). For $\alpha\in\Omega(r)$ we set
\begin{equation}
  (R^{(r)}_\alpha f)(z)=\begin{cases}\,\,\dfrac{f(z)}{r(z)-\alpha}-\sum_{n=1}^N\dfrac{f(w_n)}{r^\prime(w_n)(z-w_n)},\quad r(z)\not=\alpha,\\
    \\
    \,\, \sum_{n=1}^N\dfrac{f^\prime(w_n)}{r^\prime(w_n)}, \hspace{3.9cm} z\in r^{-1}\left\{\alpha\right\}=\left\{w_1,\ldots, w_N\right\},
  \end{cases}
  \label{newtotoche}
\end{equation}
where $f$ is assumed analytic in a neighborhood of the set $r^{-1}\left\{\alpha\right\}$ (the use of a lower case $f$ for \eqref{newtotoche} and of an upper case $F$ for \eqref{totoche} is on purpose, and is motivated by expressions such as \eqref{fF}).
We note that the first formula in \eqref{newtotoche} can be rewritten as (for $\alpha\in\mathbb C$ such that  $r(\infty)\not=\alpha$ and $z\not\in r^{-1}\left\{\alpha\right\}$):
\begin{equation}
\label{newtotoche2}
\begin{split}
(R^{(r)}_\alpha f)(z)&=\frac{f(z)}{r(\infty)-\alpha}+\sum_{n=1}^N\frac{1}{r^\prime(w_n)}\dfrac{f(z)-f(w_n)}{z-w_n}\\
&=\frac{f(z)}{r(\infty)-\alpha}+\sum_{n=1}^N \frac{1}{r^\prime(w_n)}(R_{w_n}f)(z).
\end{split}
\end{equation}
The operators \eqref{newtotoche} satisfy the resolvent identity
\begin{equation}
  \label{reso45}
R_\alpha^{(r)}-R_\beta^{(r)}=(\alpha-\beta)R_\alpha^{(r)}R_\beta^{(r)},\quad \alpha,\beta\in\Omega(r),
\end{equation}
see \cite{AJLV15}. Another slightly different proof  can be obtained using \eqref{newtotoche2} since the classical resolvent operators satisfy the resolvent identity; see Corollary \ref{011821}
below. When these operators have a trivial kernel, the underlying operator models are now of the form \eqref{newmodel} below, that is
\[
(Tf)(z)=r(z)f(z)+h_f(z)
\]
where $h_f(z)$ is a rational function of $z$, linear in $f$.\\

In the present paper we study the counterpart of the de Branges and de Branges -Rovnyak spaces when the classical backward-shift operators are replaced by the operators \eqref{newtotoche}. For instance
\eqref{wdc} will now be replaced by
\begin{equation}
  \label{wdccdw}
    \left[ R^{(r)}_\alpha f,g\right] -\left[ f, R_\beta^{(r)} g\right]  -(\alpha-\overline{\beta})\left[ R^{(r)}_\alpha f,R^{(r)}_\beta g\right]=i(\alpha-\overline{\beta})\sum_{i,j=1}^N
  \frac{g(v_j)^*}{\overline{r^\prime(v_j)}}J\frac{f(u_i)}{r^\prime(u_i)}\frac{1}{u_i-\overline{v_j}}
\end{equation}
where the $u_i$ (resp. the $v_j$) are the pre-images of $\alpha$ (resp. of $\beta$) under $r$. We note that the right hand side of \eqref{wdccdw} is adapted from the Riemann surface case;
see \cite[Theorem 5.1, p. 316]{av3}.\\

To describe our results we need some more notation and definitions.

\begin{definition}
  \label{Statespace}
    The linear span of the functions  $z\mapsto \frac{r(z)-r(a)}{z-a}$ where $a$ runs
through the points of analyticity of $r$, will be called the state space associated with $r$ and will be denoted by $\mathfrak L(r)$.
\end{definition}

The space $\mathfrak L(r)$ is finite dimensional, as we recall in Lemma \ref{rtyuiop}, and we denote by $Z_r(z)$ a $1\times N$ matrix with entries a basis of $\mathfrak L(r)$:
\begin{equation}
  Z_r(z)=\begin{pmatrix}e_1(z)&\cdots &e_N(z)\end{pmatrix}.
  \label{zzzzz}
\end{equation}

\begin{example}
  For $r(z)=z+\frac{1}{z}$, the space $\mathfrak L(r)$ is spanned by the functions $1$ and $1/z$ since
  \[
\frac{r(z)-r(a)}{z-a}=1-\frac{1}{za}
\]
for this $r$.
\label{zz-1}
  \end{example}

\begin{remark}{\rm
    When $r(z)=z$, the corresponding $Z_r(z)$ can be taken to be $1$, and the operator reduces to the classical backward shift operators \eqref{totoche}.}
  \end{remark}

  We will be interested in functions of the form
  \begin{equation}
    f(z)=(Z_r(z)\otimes I_s)F(r(z)),
    \label{tyuiop}
    \end{equation}
    where $F$ is analytic in a neighborhood of the origin, and a characterization of these functions is given in Theorem \ref{234432}.\smallskip

Condition \eqref{wdccdw} is now equivalent (see Theorem \ref{thm4.8} for a precise statement) to the reproducing kernel of the space to be of the form
\begin{equation}
 ( Z_r(z)\otimes I_s)Y^{-1}\frac{J_0-\Theta(r(z))J_0\Theta(r(w))^*}{-i(r(z)-\overline{r(w)})}Y^{-*}(Z_r(w)\otimes I_s)^*,
\end{equation}
where $\Theta$ is a $\mathbb C^{rs\times rs}$-valued function analytic in a neighborhood of the origin, $Y\in\mathbb C^{rs\times rs}$ is invertible, $J_0\in
\mathbb C^{rs\times rs}$ is a signature matrix, and the latter two depend only on $Z_r$ and $J$.\\

Another question which we are considering here is related to the Cuntz relations. With the above notation, we look for pairs of positive definite kernels $(K,K_0)$ such that, with
$K_0(r)(z,w)=K_0(r(z),r(w))$, we have the orthogonal decomposition (the functions $e_j$ have been defined by \eqref{zzzzz})
\begin{equation}
  \mathfrak H(K)=\oplus e_j\mathfrak H(K_0(r)).
\end{equation}

The weighted multiplication operators $T_j:\,\,\mathfrak H(K_0)\longrightarrow \mathfrak H(K)$, $j=1,\ldots, N$, defined by
\begin{equation}
(T_jg)(z)=e_j(z)g(r(z)),\quad j=1,\ldots, N
  \end{equation}
  satisfy then the Cuntz relations.\\

  An important tool in our results is a representation theorem for analytic functions, characterizing functions of the form $Z_r(z)F(r(z))$ (the case where
  $r(z)$ is a polynomial has been considered in \cite{AJLV15}).\\

  The paper contains ten sections besides the introduction. Sections 2-7 form the first part of the paper. After the preliminary Section 2, in which we discuss some aspects of realizations of rational functions, the main general results are presented in Sections 3-7. Specifically, in Section 3 we study, given a rational function, an associated symmetric matrix also proving a technical result which may be of independent interest. Section 4 contains a main result, namely a representation theorem, see Theorem 4.1, which generalizes the polynomial case and some particular cases of rational functions already studied in the literature. Section 5 discusses how our results are connected with Cuntz relations. Section 6 studies subspaces invariant under the action of
  suitable generalized backward shift operators, while Section 7 moves to related operator models.
Sections 8-11 are devoted to the counterpart of de Branges-Rovnyak spaces in the present setting. We study the Pontryagin spaces with assigned reproducing kernels in the line and in the circle case, we present two structure theorems, and we consider also the case of negative squares and generalized Nevanlinna functions.
\section{Preliminary results}
\setcounter{equation}{0}

In this section, which is of a preliminary nature, we discuss various aspects of realization theory of rational functions.
We mention in particular Corollary \ref{corocoro}, which may be of independent interest, and Corollary \ref{corocoro1},
which allows us to make the connection with the classical case where $r(z)=z$.
The first result follows from a partial fraction expansion; we write out the details since \eqref{newtotoche4} plays a central role in the sequel. We note that it may be that
$r(\infty)\in{\rm ran}\, r$, for example  $r(z)=1+\dfrac{z-1}{z^2+1}$. We have $r(1)=r(\infty)=1$.

\begin{proposition}
  Let $r$ be a rational function of degree $N$ and let $\alpha\in\Omega(r)$. We have:
\begin{equation}
  \label{newtotoche4}
  \frac{1}{r(z)-\alpha}=\frac{1}{r(\infty)-\alpha}
  +\sum_{n=1}^N\frac{1}{r^\prime(w_n)(z-w_n)},
  \end{equation}
  which reduces to
\begin{equation}
  \label{newtotoche456}
    \frac{1}{r(z)-\alpha}=\sum_{n=1}^N\frac{1}{r^\prime(w_n)(z-w_n)},
  \end{equation}
  when  $r(\infty)=\infty$.
  \end{proposition}

  \begin{proof} The zeros $w_1,\ldots, w_N$ cannot be zeros of $q$, even if we interpret $r(w_n)=\alpha$ as $p(w_n)=\alpha q(w_n)$. If $q(w_n)=0$ then $p(w_n)=0$, but $p$ and $q$ are assumed prime
    together. From $\alpha=\frac{p(w_n)}{q(w_n)}$ we get
    \begin{equation}
      \label{rprime}
      \begin{split}
        \frac{p^\prime(w_n)-\alpha q^\prime(w_n)}{q(w_n)}&=\frac{p^\prime(w_n)- \dfrac{p(w_n)}{q(w_n)}q^\prime(w_n)}{q(w_n)}\\
        &=\frac{p^\prime(w_n)q(w_n)-p(w_n)q^\prime(w_n)}{q(w_n)^2}\\
        &=r^\prime(w_n).
\end{split}
      \end{equation}
   Partial fraction expansion gives
    \[
      \begin{split}
        \frac{1}{r(z)-\alpha}&=\frac{1}{r(\infty)-\alpha}+\frac{q(z)}{p(z)-q(z)\alpha}\\
        &=\frac{1}{r(\infty)-\alpha}+
\sum_{n=1}^N\frac{q(w_n)}{p^\prime(w_n)-\alpha q^\prime(w_n)}\frac{1}{z-w_n}\\
        &=\frac{1}{r(\infty)-\alpha}+
\sum_{n=1}^N\frac{q(w_n)}{p^\prime(w_n)-\frac{p(w_n)}{q(w_n)}q^\prime(w_n)}\frac{1}{z-w_n}\\
        &=\frac{1}{r(\infty)-\alpha}+\sum_{n=1}^N\frac{1}{r^\prime(w_n)(z-w_n)},
        \end{split}
      \]
      where we have used \eqref{rprime}.
    \end{proof}

    We note that if $r(\infty)=0$ then $0\not\in \Omega(r)$.\\

From the previous result we deduce:

\begin{corollary}
  \label{011821}
  Let $r$ be a rational function of degree $N$ and $\alpha,\beta\in\Omega(r)$. Then \eqref{reso45} follows from \eqref{newtotoche2}.
\end{corollary}

\begin{proof}
  Let $w_1,\ldots, ,w_N$ be the roots (which are pairwise different since $\alpha\in\Omega(r)$) of the equation $r(z)=\alpha$, and let $v_1,\ldots, v_N$ be the roots (pairwise different since
  $\beta\in\Omega(r)$) of $r(v_m)=\beta$. With
  \[
    c=\frac{1}{r(\infty)-\alpha},\quad d=\frac{1}{r(\infty)-\beta},\quad c_n=\frac{1}{r^\prime(w_n)},\quad d_n=\frac{1}{r^\prime(v_n)},\,\, n=1,\ldots, N ,
  \]
 we can write (since the operator \eqref{totoche} satisfy the resolvent identity, and using \eqref{newtotoche4}):
\[
  \begin{split}
    R^{(r)}_\alpha R^{(r)}_\beta&=\left(c+\sum_{n=1}^Nc_nR_{w_n}\right)\left(d+\sum_{m=1}^Nd_mR_{v_m}\right)\\
    &=cR_\beta^{(r)}+d\sum_{n=1}^Nc_ndR_{w_n}+\sum_{n,m=1}^N c_nd_mR_{w_n}R_{v_m}\\
    &=cR_{\beta}^{(r)}+d(R^{(r)}_\alpha-c)+\sum_{n,m=1}^N c_nd_m\frac{R_{w_n}-R_{v_m}}{w_n-v_m}\\
    &=cR_{\beta}^{(r)}+d(R^{(r)}_\alpha-c)+\sum_{n=1}^N c_nR_{w_n}\left(\sum_{m=1}^Nd_m\frac{1}{w_n-v_m}\right)-\\
    &\hspace{5mm}-\sum_{m=1}^N d_mR_{v_m}\left(\sum_{n=1}^Nc_n\frac{1}{w_n-v_m}\right)\\
    &=cR_{\beta}^{(r)}+d(R^{(r)}_\alpha-c)+(R_\alpha^{(r)}-c)\left(\frac{1}{\alpha-\beta}-d\right)-\\
    &\hspace{5mm}-(R^{(r)}_\beta-d)\left(-\frac{1}{\beta-\alpha}+c\right)\\
    &=\frac{R_\alpha^{(r)}-R_\beta^{(r)}}{\alpha-\beta}+cd+\frac{d-c}{\alpha-\beta}\\
        &=\frac{R_\alpha^{(r)}-R_\beta^{(r)}}{\alpha-\beta}
  \end{split}
\]
since, by direct computation, we have $cd+\dfrac{d-c}{\alpha-\beta}=0.$
\end{proof}

For the next lemma, we use realization theory for rational functions; see e.g. \cite{bgk1, MR2663312,MR1393938}. We follow the analysis in \cite{ad9}. We will content ourselves with recalling the
following facts. First, a $\mathbb C^{r\times s}$-valued function $M$ is rational (i.e. has rational components) if and only if the linear span of the functions $R_aMc$ is finite dimensional, where $a$ runs through the
regular points of $M$ and $c$ runs through $\mathbb C^s$. Next, every rational function regular at the origin can be written in the form
\begin{equation}
  M(z)=H+zG(I-zT)^{-1}F,
\end{equation}
where $H=M(0)$ and $G,T,F$ are matrices of appropriate sizes.
One way to build a realization (in terms of linear operators rather than matrices) uses the space $\mathfrak L(M)$ and is called
the backward-shift realization. More generally, if $M$ is regular at a point $a\in\mathbb C$ there exists a realization of the form
\begin{equation}\label{Ma}
M(z)=H+(z-a)G(T_1-zT_2)^{-1}F
\end{equation}
with $\det (T_1-aT_2)\not=0$. The realization is minimal if the size of the matrices $T_1$ and $T_2$ is minimal. Assuming $T_1,T_2\in\mathbb C^{N\times N}$, minimality is equivalent to
having both controllability,
\begin{equation}
  \label{obs}
\bigcap_{z\in\mathcal U} \ker G(T_1-zT_2)^{-1}=\left\{0\right\}
\end{equation}
and observability
\[
  \bigcup_{z\in\mathcal U}{\rm ran}\, (T_1-zT_2)^{-1}F=\mathbb C^N,
\]
where $\mathcal U$ is a neighborhood of $a$. One can in fact consider a finite number of different points $w_1,\ldots, w_N$, since we are in the rational dimensional case.
For instance, condition \eqref{obs} can be rewritten as
\[
  c\in\mathbb C^N\,\,:\,\,G(T_1-wT_2)^{-1}c
  =0,\quad z\in\mathcal U\,\,\Longrightarrow \,\, c=0
\]
but the function $G(T_1-wT_2)^{-1}c$ is rational of degree less or equal to $N$.\smallskip

We now explicit the backward shift realization:

\begin{proposition}
Let $M$ be a ${\mathbb C}^{p\times q}$--valued rational function with
associated space ${\mathfrak L}(M)$ and let $a$ be a point of analyticity of
$M$. Then,
\begin{equation}
M(z)=H+(z-a)G(I_{\mathfrak L(M)}-(z-a)A)^{-1}F
\end{equation}
where the operator matrix
\begin{equation}
\left(\begin{array}{cc}A&F\\G&H\end{array}\right)\quad:\quad
\left(\begin{array}{c}{\mathfrak L}(M)\\ {\mathbb C}^q\end{array}
\right)\Longrightarrow
\left(\begin{array}{c}{\mathfrak L}(M)\\ {\mathbb C}^p\end{array}\right)
\end{equation}
is defined by (with $c\in\mathbb C^q$)
\begin{eqnarray}
Af(z)&=&\frac{f(z)-f(a)}{z-a}\\
Fc&=&\frac{M(z)-M(a)}{z-a}c\\
Gf&=&f(a)\\
Hc&=&M(a)c.
\end{eqnarray}
\end{proposition}

It is called a realization centered at $a$. It corresponds to
\[
T_1=I_{\mathfrak L(M)}+aA\quad {\rm and}\quad T_2=A
\]
in  \eqref{Ma}.
When $a=0$ the above realization is called the
{\sl backwards shift realization}.\\

For completeness we mention and prove:

\begin{proposition}
  \label{minizero}
  Assume the realization \eqref{Ma} minimal. Then, the set of points $\mathcal U(M)$ where $M$ is defined coincides with the set of points $U$ for
  which $\det (T_1-zT_2)\not=0$.
\end{proposition}

\begin{proof}
  One direction is clear, $U\subset \mathcal U(M)$,
  and does not use minimality. If $\det (T_1-zT_2)\not=0$, then $M(z)$ is well defined at that point. To study the converse
we note the formula
\begin{equation}
\frac{M(z)-M(w)}{z-w}=G(T_1-zT_2)^{-1}(T_1-aT_2)(T_1-wT_2)^{-1}F,
\end{equation}
valid {\sl a priori} for $z,w\in U$.\smallskip

Taking now $z\in\mathcal U(M)$ and $w_1,\ldots, w_N$ such that the union of the ranges of $(T_1-w_nT_2)^{-1}F$ span $\mathbb C^N$ we see that
\[
G(T_1-zT_2)^{-1}(T_1-aT_2)\begin{pmatrix}(T_1-w_1T_2)^{-1}F&\cdots&(T_1-w_NT_2)^{-1}F\end{pmatrix}
\]
is well defined at $z$, and so is $G(T_1-zT_2)^{-1}$. Take now $w_1,\ldots, w_N$ such that the left range of the matrices
\[
G(T_1-w_nT_2)^{-1},\quad n=1,\ldots, N,
\]
is all of $\mathbb C^N$, we see that $(T_1-zT_2)^{-1}$ is well defined.
\end{proof}

\begin{remark}
For the interested reader, we mention that
the case of functions analytic at $\infty$ corresponds to realizations of the form
\begin{equation}
  M(z)=H+G(zI-T)^{-1}F.
\end{equation}
\end{remark}

In the statement of next result, observability of the pair $(G,T)$  reduces to
\begin{equation}
  \bigcap_{\ell=0}^\infty \ker GT^\ell=\left\{0\right\}.
  \end{equation}
  We give a proof of the lemma for completeness.
\begin{lemma}
  \label{rtyuiop}
  The linear span $\mathfrak L(r)$ of the functions  $z\mapsto \dfrac{r(z)-r(a)}{z-a}$ where $a$ runs through $\mathcal U(r)$, is finite dimensional, of dimension $N$, where $N={\rm deg}\, r$.
  Assume $r$ analytic at the origin, and let $e_1,\ldots, e_N$ be a basis of $\mathfrak L(r)$.
  There exists an observable pair $(G,T)\in\mathbb C^{1\times N}\times \mathbb C^{N\times N}$ such that
  \begin{equation}
\label{Zr6789}
    Z_r(z)\stackrel{\rm def.}{=}\begin{pmatrix}e_1(z)&\cdots &e_N(z)\end{pmatrix}=G(I_N-zT)^{-1}
  \end{equation}
  Furthermore, if $r$ is real, meaning $r(z)=\overline{r(\overline{z})}$, the matrices $G$ and $T$ may be chosen to have real components.
\end{lemma}
\begin{proof}
  The claim is invariant under a Moebius transform, and we will assume that $0$ is not a pole of $r$. We can write $r$ in a realized form as
  \begin{equation}
    \label{real-r}
r(z)=d+zG(I_N-zT)^{-1}b,
\end{equation}
and thus
\begin{equation}
  \label{bghjkl}
\frac{r(z)-r(w)}{z-w}=G(I_N-zT)^{-1}(I-wT)^{-1}b,
\end{equation}
from which the result follows.
  For the last claim see \cite{abgr1}. A more direct way is to write real realizations of $p$ and $q$ and get a realization for $p/q$.
    \end{proof}

\begin{remark}{\rm We note that two different choices of basis will amount to two different $Z_r$ differing by an invertible multiplicative matrix on the left, corresponding to the change of basis.
    In the paper we always stick to a pre-assigned choice, unless stated otherwise.}
\end{remark}

\begin{example}
    When $r(z)=p(z)$ we can choose $e_n(z)=z^{n-1}$, $n=1,\ldots, N$ and then in \eqref{real-r}, we have $G=\begin{pmatrix}1&1&\cdots& 1\end{pmatrix}$ and $T$ is the nilpotent matrix
    \[
      T=\begin{pmatrix}0&1&0&&\cdots &0\\
                       0&0&1&0&\cdots&0\\
                       0&& &\ddots&&\\
\vdots&& &&&\vdots\\
0&0 &&\cdots&0&1\\
0&0 &&\cdots&0&0\end{pmatrix}
      \].
    \end{example}

    \begin{example}
      When $r(z)$ is a finite Blaschke product of the open unit disk, with pairwise different zeros, we can write
      \[
r(z)=\prod_{n=1}^N\frac{z-a_n}{1-z\overline{a_n}}.
\]
We have, as sets, $\mathfrak L(r)=\mathbf H_2(\mathbb D)\ominus r\mathbf H_2(\mathbb D)$, and we can choose $e_n(z)=\frac{1}{1-z\overline{a_n}}$, so that here too
$G=\begin{pmatrix}1&1&\cdots& 1\end{pmatrix}$ and $T$ is the diagonal matrix ${\rm diag}\,(\overline{a_1},\overline{a_2},\ldots, \overline{a_N})$.
\end{example}

\begin{lemma}
  \label{rtyuiop1}
  Let $r(z)=\frac{p(z)}{q(z)}$ be a rational function of degree $N\geq 1$.
  Elements of $\mathfrak L(r)$ are exactly rational functions of the form $\frac{m(z)}{q(z)}$, where $m$ is a polynomial of degree at most $N-1$. In particular,
  if $m(z)\not\equiv 0$, then it cannot have more than $N-1$ zeros.
\end{lemma}

\begin{proof} This follows from
  \[
    \begin{split}
      \frac{r(z)-r(a)}{z-a}&=\frac{\left(\dfrac{p(z)q(a)-q(z)p(a)}{z-a}\right)}{q(z)q(a)}.
      \end{split}
      \]
\end{proof}

\begin{corollary}
  Let $r(z)$ be a rational function of degree $N$, with numerator of the form $\prod_{i=1}^M(z-w_i)^{m_i}$ with $\sum_{i=1}^Mm_i=N$, and
  let $\xi\in\mathbb C^N$ be such that
  \begin{equation}
    \label{zeroes321123}
    Z_r^{(j)}(w_i)\xi=0,\quad i=1,\ldots, M,\,\, j=0,\ldots, m_i-1.
\end{equation}
Then $\xi=0$.
\label{coro210}
\end{corollary}

\begin{proof}
  Indeed, the function $z\mapsto Z_r(z)\xi$ belongs to $\mathfrak L(r)$. Condition \eqref{zeroes321123} expresses that it has zeros of total multiplicity
  $\sum_{i=1}^M m_i=N$, and so by the previous result we have $Z_r(z)\xi\equiv 0$, and so $\xi=0$ since the entries of $Z_r(z)$ are linearly independent.
\end{proof}

From this corollary we obtain:

\begin{theorem}
  \label{unique123321}
  Let $r(z)$ be a rational function of degree $N$, with numerator of the form $\prod_{i=1}^M(z-w_i)^{m_i}$ with $\sum_{i=1}^Mm_i=N$. Assume that
  \begin{equation}
    \label{tyu}
    Z_r(z)F(r(z))\equiv 0.
  \end{equation}
  Then, $F\equiv 0$.
\end{theorem}

\begin{proof}
  We have
  \begin{eqnarray*}
    (    Z_r(z)F(r(z)))^{(1)}&=&(Z_r(z))^{(1)}F(r(z))+Z_r(z)r^{(1)}(z)F^{(1)}(r(z))\\
    (    Z_r(z)F(r(z)))^{(2)}&=&(Z_r(z))^{(2)}F(r(z))+2(Z_r(z))^{(1)}r^{(1)}(z)F^{(1)}(r(z))+\\
    \nonumber
    &&\hspace{1mm}+Z_r(z)r^{(2)}(z)F^{(1)}(r(z))+Z_r(z)(r^{(1)}(z))^2F^{(2)}(r(z)),
                                           \nonumber
  \end{eqnarray*}
  and similarly for derivatives of higher order. Using that $r^{(j)}(w_i)=0$ for $i=1,\ldots, M,\,\, j=0,\ldots, m_i-1$, we get
  \begin{equation}
    \label{zzeroes321123}
    Z_r^{(j)}(w_i)F(0)=0,\quad i=1,\ldots, M,\,\, j=0,\ldots, m_i-1,
  \end{equation}
  so that $F(0)=0$ in view of Corollary \ref{coro210}. We write $F(z)=zF_1(z)$. Equation \eqref{tyu} becomes
  \[
    Z_r(z)F_1(r(z))\equiv 0.
  \]
  Reiterating the previous argument on $Z_r(z)F_1(r(z))$ we have that $F^{(1)}(0)=F_1(0)=0.$ Iterating we get that the coefficients of
  the Taylor series expansion of $F$ at the origin all vanish, and hence $F\equiv 0$.
\end{proof}

The results in the next corollary play a key role in the sequel; in particular see the proof of Theorem \ref{tm2021} in Section \ref{sec2021} on the Cuntz relations. Note that the left hand side of  \eqref{totototo} below is of the form \eqref{tyuiop}.

\begin{corollary}
\label{corocoro}
In the previous notation, let $\alpha\in\Omega(r)$. Then, the matrix $\mathsf G\in\mathbb C^{N\times N}$ defined by
\begin{equation}
g_{nj}=(e_n(w_j))_{n,j=1}^N
\end{equation}
is invertible.
Let $F_1,\ldots, F_N$ be analytic in a connected neighborhood $\Omega$ of $\alpha$. Then
\begin{equation}
  \label{totototo}
  \sum_{n=1}^Ne_n(z)F_n(r(z))\equiv 0,\quad z\in r^{-1}(\Omega),
  \end{equation}
if and only if $F_n(z)\equiv 0$ for $z\in\Omega$ and $n=1,\ldots, N$.
\end{corollary}

\begin{proof} Let $c=\begin{pmatrix}c_1&c_2&\cdots &c_N\end{pmatrix}^t\in\mathbb C^N$, and assume that $\mathsf Gc=0$. Let $g(z)=\sum_{n=1}^N c_ne_n(z)$. The function $g\in\mathfrak L(r)$, and
the condition $\mathsf Gc=0$ implies that $g(w_i)=0$ for $i=1,\ldots, N$. By the previous lemma $g(z)\equiv 0$ and hence $c=0$ since $e_1,\ldots, e_N$ are linearly independent.\smallskip

Assume now that \eqref{totototo} holds. Setting $z=w_i$, $i=1,\ldots, N$ we have that
\[
\sum_{n=1}^N e_n(w_i)F_n(\alpha)=0,\quad i=1,\ldots, N
\]
and so $F_1(\alpha)=\cdots =F_N(\alpha)$ in view of the first part of the corollary. For $n=1,\ldots, N$ write now $F_n(z)=(z-\alpha)G_n(z)$, where $G_n$ is analytic in $\Omega$.
We have
\[
  \sum_{n=1}^Ne_n(z)r(z)G_n(r(z))\equiv 0,\quad z\in r^{-1}(\Omega),
\]
and so
\[
  \sum_{n=1}^Ne_n(z)G_n(r(z))\equiv 0,\quad z\in r^{-1}(\Omega).
\]
Reiterating the previous argument we get $G_1(\alpha)=\cdots =G_N(\alpha)=0$. We then get that the Taylor series around $\alpha$ for $F_1,\ldots, F_N$ vanish identically, and hence the functions
$F_1,\ldots, F_N$ vanish identically in $\Omega$ since the latter is connected.
\end{proof}

\begin{proposition}
Let $r$ be a rational function of degree $N$ and let $\alpha\in\Omega(r)$. Then, for $f\in\mathfrak L(r)$, we have
\begin{equation}
  \label{bastille}
          \frac{f(z)}{r(z)-\alpha}=\sum_{n=1}^N\frac{f(w_n)}{r^\prime(w_n)(z-w_n)}
\end{equation}
and, in particular,
        \begin{equation}
          \frac{Z_r(z)}{r(z)-\alpha}=\sum_{n=1}^N\frac{Z_r(w_n)}{r^\prime(w_n)(z-w_n)}.
\label{2.5}
        \end{equation}
                  \end{proposition}

                  \begin{proof}
It is enough to prove \eqref{bastille} for $f(z)=\dfrac{r(z)-r(\beta)}{z-\beta}$ since $\mathfrak L(r)$ is finite dimensional and spanned by such functions.
Since $\dfrac{q(\beta)p(z)-p(\beta)q(z)}{z-\beta}$ is a polynomial of degree strictly less that $N$, we then have the partial fraction expansion:
\[
  \begin{split}
    \frac{f(z)}{r(z)-\alpha}&=\frac{\left(\dfrac{q(\beta)p(z)-p(\beta)q(z)}{z-\beta}\right)}{q(\beta)(p(z)-\alpha q(z))}\\
    \\
    &=\sum_{n=1}^N\frac{\left(\dfrac{q(\beta)p(w_n)-p(\beta)q(w_n)}{w_n-\beta}\right)}{q(\beta)(p^\prime(w_n)-\alpha q^\prime(w_n))}\\
\\
    &=\sum_{n=1}^N\frac{\left(\dfrac{q(\beta)p(w_n)-p(\beta)q(w_n)}{q(w_n)(w_n-\beta)}\right)}{\dfrac{q(\beta)(p^\prime(w_n)-\alpha q^\prime(w_n))}{q(w_n)}}\\
\\
&=\sum_{n=1}^N\frac{f(w_n)}{r^\prime(w_n)(z-w_n)},\\
&
  \end{split}
\]
where we have used another time \eqref{rprime}.
\end{proof}

The following corollary shows that the operators $R^{(r)}_\alpha$ are well suited for functions of the form \eqref{tyuiop}

\begin{corollary}
\label{corocoro1}
Let $r$ be a rational function of degree $N$.
Let $w_1(\A),\ldots, w_N(\A)$ the roots, assumed simple, of the equation $r(z)=\alpha$, and let $f$ be analytic in a neighborhood of $\left\{w_1(\A),\ldots, w_N(\A)\right\}$ and of the form
$f(z)=Z_r(z)F(r(z))$, where the function $F$ is $\mathbb C^N$-valued and
analytic in a neighborhood of the point $\alpha$. Then,
\begin{equation}
  \label{formulacoro}
(R^{(r)}_\alpha f)(z)=Z_r(z)(R_\alpha F)(r(z)).
\end{equation}
\end{corollary}

\begin{proof}
  We can write
\[
\begin{split}
(R^{(r)}_\alpha f)(z)&=\frac{Z_r(z)F(r(z))}{r(z)-\alpha}-\sum_{n=1}^N\frac{Z_r(w_n)F(r(w_n))}{r^\prime(w_n)(z-w_n)}\\
&=\frac{Z_r(z)F(r(z))}{r(z)-\alpha}-\sum_{n=1}^N\frac{Z_r(w_n)F(\alpha)}{r^\prime(w_n)(z-w_n)}\\
&=\frac{Z_r(z)F(r(z))}{r(z)-\alpha}-\sum_{n=1}^N\frac{Z_r(w_n)}{r^\prime(w_n)(z-w_n)}F(\alpha)\\
&=\frac{Z_r(z)F(r(z))}{r(z)-\alpha}-\frac{Z_r(z)}{r(z)-\alpha}F(\alpha)\\
&=Z_r(z)(R_\alpha F)(r(z)),
\end{split}
\]
where we have used \eqref{2.5} to conclude.
\end{proof}

\section{The associated symmetric matrix}
\setcounter{equation}{0}

To a given rational function we can associate a symmetric matrix which plays an important role in the sequel.
The following lemma seems to be new and possibly of independent interest.

\begin{lemma}
  \label{lemm26}
  Let $r$ be a rational function of degree $N$,  $\alpha\in\Omega(r)$ and let $w_1(\A),\ldots, w_N(\A)$  be the (simple)
  roots of $r(z)=\alpha$. Let $f,g\in\mathfrak L(r)$, then the sum
  \begin{equation}
   \sum_{n=1}^N \frac{f(w_n(\A))g(w_n(\A))}{r^\prime(w_n(\A))}
 \end{equation}
  does not depend on $\alpha$.
\end{lemma}

\begin{proof} It is enough to consider $f$ and $g$ respectively equal to
  \[
f(z)=\frac{r(z)-r(a)}{z-a} \quad{\rm and}\quad g(z)=\frac{r(z)-r(b)}{z-b},
\]
where $a,b$ are points of analyticity of $r$.  We note that the poles of $r$ are removable singularities of
\[
  \frac{f(z)g(z)}{(r(z)-\alpha)r(z)}
  =\frac{\left(\dfrac{p(z)q(a)-q(z)p(a)}{z-a}\right)\left(\dfrac{p(z)q(b)-q(z)p(b)}{z-b}\right)}{(p(z)-\alpha q(z))p(z)q(a)q(b)},
\]
and so, for $R>0$ large enough (depending on $\alpha$), such that all the poles of the integrand are inside $|z|\le R$,
\[
\frac{1}{2\pi i}\int_{|z|=R}\frac{f(z)g(z)}{(r(z)-\alpha)r(z)}dz=0
\]
since the difference between the degree of the denominator  and of the numerator is at least $2$.
So, for such $R$:
%\[
%\sum_{n=1}^N\frac{f(u_n)g(u_n)}{(r^\prime(u_n)r(u_n)}}
%\]
\begin{equation}
  \label{equality67890}
\frac{1}{2\pi i}\int_{|z|=R}\frac{f(z)g(z)}{r(z)-\alpha}dz=\frac{1}{2\pi i}\int_{|z|=R}\frac{f(z)g(z)}{r(z)}dz.
\end{equation}
    Furthermore
    \begin{equation}
      \label{form789}
  \frac{f(z)g(z)}{r(z)-\alpha}
  =\frac{\left(\dfrac{p(z)q(a)-q(z)p(a)}{z-a}\right)\left(\dfrac{p(z)q(b)-q(z)p(b)}{z-b}\right)}{(p(z)-\A q(z))q(z)q(a)q(b)}
\end{equation}

and in particular, for $\alpha=0$,
    \[
  \frac{f(z)g(z)}{r(z)}
  =\frac{\left(\dfrac{p(z)q(a)-q(z)p(a)}{z-a}\right)\left(\dfrac{p(z)q(b)-q(z)p(b)}{z-b}\right)}{p(z)q(z)q(a)q(b)}.
\]
It follows from these two equalities that to evaluate both sides of \eqref{equality67890} we need to compute the residues of the corresponding functions at the points
$w_n(\alpha),w_n(=w_n(0))$ and at the zeros of $q$.
Since $\alpha\in\Omega(r)$, we have
\[
  {\rm Res\,}\left(  \frac{f(z)g(z)}{r(z)-\alpha},w_n(\alpha)\right)=\frac{f(w_n(\A))g(w_n(\A))}{r^\prime(w_n(\A))},\quad n=1,\ldots, N.
\]
The key to the proof of the lemma is that the residues at the zeros of $q$ do not depend on $\alpha$, as we now show.
Let $z_1,\ldots, z_U$ be the zeros of $q$, and assume $z_u$ be a zero of order $m_u$: $q(z)=(z-z_u)^{m_u}\pi_u(z)$, where $\pi_u(z)$ is a polynomial not vanishing at
  $z_u$.
  We rewrite \eqref{form789} as
  \begin{equation}
    \frac{f(z)g(z)}{r(z)-\alpha}=
    \frac{G(z)}{(p(z)-\alpha q(z))(z-z_u)^{m_u}}
    \end{equation}
    where $G$ is independent of $\alpha$ and is analytic at the point $z_u$ and $G(z_u)\not=0$. From the formula
    \[
      {\rm Res\,}\left(  \frac{f(z)g(z)}{r(z)-\alpha},z_u\right)=\frac{1}{(m_u-1)!}\left(\frac{G(z)}{(p(z)-\alpha q(z))}\right)^{(m_u-1)}(z_u)
    \]
    we see that the residue involves only derivatives of $\alpha q$ of order less or equal to $m_u-1$, evaluated at the point $z_u$, and hence does not depend on $\alpha$
    since
\[
q(z_u)=\cdots=q^{(m_u-1)}(z_u)=0.
\]

Thus, we have for large $R$
    \[
      \frac{1}{2\pi i}\int_{|z|=R}\frac{f(z)g(z)}{r(z)-\alpha}dz=  \sum_{n=1}^N \frac{f(w_n(\A))g(w_n(\A))}{r^\prime(w_n(\A))}+M_1
\]
    where $M_1$ denotes the sum of the residues of the function $\frac{f(z)g(z)}{r(z)-\alpha}$ evaluated at the zeros of $q$, and
    \[
      \frac{1}{2\pi i}\int_{|z|=R}\frac{f(z)g(z)}{r(z)}dz= \frac{1}{2\pi i}\int_{|z|=R}\frac{f(z)g(z)}{r(z)}dz +M_1.
    \]
It follows that the sum
    \[
      \sum_{n=1}^N \frac{f(w_n)g(w_n)}{r^\prime(w_n)}=\frac{1}{2\pi i}\int_{|z|=R}\frac{f(z)g(z)}{r(z)}dz
    \]
    is independent of $\alpha$.\smallskip

  \end{proof}

  \begin{remark}{\rm
When the zeros of $q$ are simple, we have
\[
  M_1=\sum_{u}\frac{\left(\dfrac{p(z_u)}{(z_u-a)(z_u-b)}\right)}{q^\prime(z_u)},
\]
and when $0\in\Omega(r)$
\[
\frac{1}{2\pi i}\int_{|z|=R}\frac{f(z)g(z)}{r(z)}dz=
  \sum_{n=1}^N \frac{f(w_n)g(w_n)}{r^\prime(w_n)}.
  \]
}
\end{remark}

\begin{corollary}
Let $r$ be a {\rm real} rational function of degree $N$,  $\alpha\in\Omega(r)$ and let us assume that the roots $w_1(\A),\ldots, w_N(\A)$  of $r(z)=\alpha$ are simple.
Let $J$ be a $p\times p$ signature matrix with real entries. The sum
\begin{equation}
  \label{Pgram}
X(J,r)=\sum_{n=1}^N\frac{(Z_r(w_n)\otimes I_p)^tJ(Z_r(w_n)\otimes I_p)}{r^\prime(w_n)}\in\mathbb R^{Np\times Np}
\end{equation}
is self-adjoint and real (i.e. symmetric), invertible, and independent of $\alpha$.
\label{amstramgram}
\end{corollary}

\begin{proof}
  In the proof we use the fact that, for any (possibly non square) matrices $A$ and $B$ of possibly  different sizes we have (see \cite{HornJohnson2} for a discussion and more information
  on the properties of the tensor, or Kronecker, product):
\[
  \begin{split}
    (A\otimes B)^t&=A^t\otimes B^t\\
    \overline{A\otimes B}&=\overline{A}\otimes \overline{B},
  \end{split}
\]
and that, in the case of invertible matrices,
\begin{equation}
  (A\otimes B)^{-1}=A^{-1}\otimes B^{-1}.
  \label{kronec-inv}
  \end{equation}
  Using \eqref{zzzzz} we deduce
  \[
    \begin{split}
      (Z_r(w_n)\otimes I_p)^tJ(Z_r(w_n)\otimes I_p)&=\begin{pmatrix}e_1(w_n)I_p\\ \vdots \\ e_N(w_n)I_p\end{pmatrix}J\begin{pmatrix}e_1(w_n)I_p& \cdots & e_N(w_n)I_p\end{pmatrix}\\
      &=\left(e_i(w_n)e_j(w_n)J\right)_{i,j=1}^N\\
      &=X(1,r)\otimes J,
\end{split}
\]
and the fact that $X(1,r)$ does not depend on $\alpha$ follows from the preceding lemma.
To prove that $X(J,r)$ is invertible it is enough to show that $X(1,r)$  is invertible. Using \eqref{kronec-inv} we then have
\[
  X(J,r)^{-1}=X(1,r)^{-1}\otimes J
\]
since $J=J^{-1}$. To prove that $X(1,r)$ is invertible we use the following readily checked fact: a $N\times N$ matrix of the form $\sum_{n=1}^N a_n^tb_n$, where the $a_1,\ldots, a_N,b_1,\ldots, b_N\in\mathbb C^N$,
    is invertible if and only if both  $a_1,\ldots, a_N$ and $b_1,\ldots, b_N\in\mathbb C^N$ are linearly independent vectors. In the case of $X(1,r)$, it is therefore enough to check that the
    vectors $Z(w_1),\ldots, Z(w_N)$ are linearly independent. Assume not. There will then be a non-zero vector
    \[
      c=\begin{pmatrix}c_1\\ \vdots \\c_N\end{pmatrix}\in\mathbb C^N
    \]
    such that
    \[
      Z(w_n)c=0,\quad n=1,\ldots, N.
    \]
    Thus the rational function $Z(z)c$ has (at least) $N$ pairwise different zeros. Since $Zc\in\mathfrak L(r)$, this cannot be in view of Lemma \ref{rtyuiop1}, unless $Z(z)c\equiv 0$. But this
    implies $c=0$ since the entries of $Z(z)$ form a basis of $\mathfrak L(r)$.\smallskip

    We now show that $X(J,r)$ is self-adjoint. Since
    \[
      (X(1,r)\otimes J)^t=X(1,r)^t\otimes J^t=X(1,r)\otimes J,
    \]
    it is enough to show that the entries of $X(1,r)$ are real.
    This follows from the fact that either the $w_n$ are real or they
    appear in pairs. In the first case the matrix $\frac{Z(w_n)^tJZ(w_n)}{r^\prime(w_n)}$ has real entries while in the second case the matrix
      \[
        \frac{Z(w_n)^tJZ(w_n)}{r^\prime(w_n)}+\frac{Z(\overline{w_n})^tJZ(\overline{w_n})}{r^\prime(\overline{w_n})}
          \]
          has real entries since $Z$ is real and $r^\prime(\overline{w_n})=\overline{r^\prime(w_n)}$.
\end{proof}

\begin{definition}
$X(J,r)$ is called the associated symmetric matrix to $r$ and $J$.
    We will write $X(J,r)=X(r)$ or $X$ when $J$ or $r$ are clear from the context.
  \end{definition}

Note that $X(1,r)$ can be written in term of a realization as:
\begin{equation}
X(1,r)=\sum_{n=1}^N\frac{(I-w_nT)^{-t}G^tG(I-w_nT)^{-1}}{r^\prime(w_n)}.
\end{equation}

\begin{corollary}
  For a Blaschke product with pairwise different zeros we have
  \begin{equation}
    \label{amstramgram1}
    X(1,r)=\left(\sum_{n=1}^N\frac{1}{r^\prime(a_n)(1-a_n\overline{a_u})
    (1-a_n\overline{a_v})}\right)_{u,v=1}^N.
  \end{equation}
\end{corollary}

\begin{proof}
We have $G=\begin{pmatrix}1&1&\cdots& 1\end{pmatrix}$ and $T$ is the diagonal matrix ${\rm diag}\,(\overline{a_1},\overline{a_2},\ldots, \overline{a_N})$.
and so
\[
    (X(1,r))_{uv}=\sum_{n=1}^N\frac{1}{r^\prime(w_n)(1-w_n\overline{a_u})(1-w_n\overline{a_v})}.
  \]
  Taking $\alpha=0$ we get $w_n=a_n$ for $n=1,\ldots, N$, from which \eqref{amstramgram1} follows by using Corollary \ref{amstramgram}.
\end{proof}

In case of a polynomial $r(z)=p(z)$ of degree $N$ and $Z_p(z)= \begin{pmatrix}1&z&\cdots &z^{N-1}\end{pmatrix}$ the above result takes the following special form:

\begin{corollary}
  Let $r(z)=p(z)=p_0+\cdots +p_Nz^{N}$, with pairwise distinct zeros $w_1,\ldots, w_N$. Then, the matrix $X(1,p)=X$ is equal to
  \begin{equation}
    X=\begin{pmatrix}0&\cdots&0&0&0&h_{N-1}\\
      0&\cdots&0&0&h_{N-1}&h_{N}\\
      0&\cdots&0&\iddots&&\vdots \\
      0&\iddots&h_{N-1}&&&\vdots\\
      0&h_{N-1}&\iddots&&&\vdots \\
            h_{N-1}&h_N&\cdots&\cdots&&h_{2N-2}
          \end{pmatrix},
          \label{xxxx}
    \end{equation}
    where
    \begin{equation}
      h_n=\frac{1}{2\pi i}\int_{|z|=R}\frac{z^ndz}{p(z)},
    \end{equation}
    with $R>\max_n |w_n|$.
\end{corollary}

\begin{proof}
  We can take $Z_p(z)=\begin{pmatrix}1&z&\cdots&z^{N-1}\end{pmatrix}$, and so
  \[
    \begin{split}
      X&=\sum_{n=1}^N\frac{\begin{pmatrix}1\\ w_n \\  \vdots\\ w_n^{N-1}\end{pmatrix}\begin{pmatrix}1& w_n &  \cdots& w_n^{N-1}\end{pmatrix}}{p^\prime(w_n)}\\
      &\\
      &=\left(\sum_{n=1}^N\frac{w_n^{k+\ell}}{p^\prime(w_n)}
      \right)_{k,\ell=0}^{N-1},
  \end{split}
\]
which is a Hankel matrix. To prove \eqref{xxxx} we now show that $X$ is lower with respect to the main off-diagonal.
Let $R_0$ be such that the roots $w_1,\ldots, w_N$ (assumed pairwise different) of $p(z)$ lie inside $B(0,R_0)$. For $R\ge R_0$ and $t=0,1,2,\ldots,$ we have
\[
  \begin{split}
    \sum_{n=1}^N\frac{w_n^t}{p^\prime(w_n)}&=\frac{1}{2\pi i}\int_{|z|=R}\frac{z^t}{p(z)-\alpha}dz\\
    &=\frac{1}{2\pi}\int_0^{2\pi}\frac{R^{t+1}e^{i(t+1)t}dt}{p(Re^{it})}
    \end{split}
  \]
  which will go to $0$ as $R\rightarrow\infty$ for $t+1<N$, that is $t\le N-2$, and hence the formula for $X$.
\end{proof}

As another example we consider the function $r(z)=z+\frac{1}{z}$ (see Example \ref{zz-1}). Then we have:

\begin{example}
  Let $\alpha\in\mathbb C$ and let $\alpha_+$ and $\alpha_-$ be the roots of
  \[
    z+\frac{1}{z}=\alpha.
  \]
  Then $Z_r(z)=\begin{pmatrix}1&\frac{1}{z}\end{pmatrix}$ and the corresponding matrix $X=\begin{pmatrix}1&0\\0&-1\end{pmatrix}$ and does not depend on $\alpha$.
\end{example}

We give a direct proof of this fact. We note that the roots satisfy $\alpha_-\alpha_+=1$.
Hence
\[
  \begin{split}
    X&=\frac{\begin{pmatrix}1\\ \frac{1}{\alpha_+}\end{pmatrix}\begin{pmatrix}1& \frac{1}{\alpha_+}\end{pmatrix}}{1-\frac{1}{\alpha_+^2}}+
    \frac{\begin{pmatrix}1\\ \frac{1}{\alpha_-}\end{pmatrix}\begin{pmatrix}1&
        \frac{1}{\alpha_-}\end{pmatrix}}{1-\frac{1}{\alpha_-^2}}\\
    &=\frac{\begin{pmatrix}1\\ \alpha_-\end{pmatrix}\begin{pmatrix}1& \alpha_-\end{pmatrix}}{1-\alpha_-^2}+
    \frac{\begin{pmatrix}1\\ \alpha_+\end{pmatrix}\begin{pmatrix}1&
        \alpha_+\end{pmatrix}}{1-\alpha_+^2}\\
    &=\begin{pmatrix}\dfrac{1}{1-\alpha_-^2}+\dfrac{1}{1-\alpha_+^2}&\dfrac{\alpha_-}{1-\alpha_-^2}+\dfrac{\alpha_+}{1-\alpha_+^2}\\
      \dfrac{\alpha_-}{1-\alpha_-^2}+\dfrac{\alpha_+}{1-\alpha_+^2}&      \dfrac{\alpha_-^2}{1-\alpha_-^2}+\dfrac{\alpha_+^2}{1-\alpha_+^2}\end{pmatrix}\\
    &=\begin{pmatrix}1&0\\
     0 &-1\end{pmatrix}
\end{split}
\]
since
\[
  \frac{1}{1-\alpha_-^2}+\frac{1}{1-\alpha_+^2}=  \frac{1}{1-\alpha_-^2}+\frac{\alpha_-^2}{\alpha_-^2-1}=1,
      \]

      \[
        \begin{split}
          \frac{\alpha_-}{1-\alpha_-^2}+\frac{\alpha_+}{1-\alpha_+^2}&=\frac{\alpha_-}{1-\alpha_-^2}+
\frac{\alpha_+\alpha_-^2}{\alpha_-^2-1}\\
          &=\frac{\alpha_-}{1-\alpha_-^2}+\frac{\alpha_-}{\alpha_-^2-1}=0
        \end{split}
      \]
      and
      \[
\frac{\alpha_-^2}{1-\alpha_-^2}+        \frac{\alpha_+^2}{1-\alpha_+^2}
=     \frac{\alpha_-^2}{1-\alpha_-^2}+     \frac{\alpha_+^2\alpha_-^2}{\alpha_-^2-1}=-1,
      \]
      where we have used repeatedly that $\alpha_+\alpha_-=1$.\\

      We also compute directly the associate symmetric matrix for a cubic.

      \begin{example}
        Let $r(z)=z^2+1/z$. Then,
        \[
          X=\begin{pmatrix}0&1&0\\ 1&0&0\\
            0&0&-1\end{pmatrix}.
        \]
      \end{example}

      Indeed, we have $Z_r(z)=\begin{pmatrix}1&z&1/z\end{pmatrix}$, and so
      \[
X=\sum_{z\,:\, r(z)=\alpha}\frac{\begin{pmatrix}1 \\z\\ \frac{1}{z}\end{pmatrix}\begin{pmatrix}1 &z&\frac{1}{z}\end{pmatrix}}{r^\prime(z)}
       \]
     and we need to compute
      \[
\sum_{n=1}^3\frac{w_n^k}{r^\prime(w_n)},\quad k=0,\pm 1,\pm 2,
      \]
      with $\alpha\in\Omega(r)$ and $r(w_n)=\alpha$ where $n=1,2,3$.\smallskip

We have for $k=0$,
\[
  \sum_{n=1}^3\frac{1}{r^\prime(w_n)}=\frac{1}{2\pi i}\int_{|s|=R}\frac{ds}{r(s)-\alpha}=\frac{1}{2\pi i}\int_{|s|=R}\frac{sds}{s^3-\alpha s+1}=0
\]
by the exactity relation. For $k=1$ we have
\[
\sum_{n=1}^3\frac{w_n}{r^\prime(w_n)}=\frac{1}{2\pi i}\int_{|s|=R}\frac{sds}{r(s)-\alpha}=\frac{1}{2\pi i}\int_{|s|=R}\frac{s^2ds}{s^3-\alpha s+1}=1,
\]
as is seen by either computing the residue at $\infty$ or letting $R\rightarrow\infty$.
For $k=2$ we have
\[
  \sum_{n=1}^3\frac{w_n^2}{r^\prime(w_n)}
  =\frac{1}{2\pi i}\int_{|s|=R}\frac{s^2ds}{r(s)-\alpha}=\frac{1}{2\pi i}\int_{|s|=R}\frac{s^3ds}{s^3-\alpha s+1}=-{\rm Res}\, (\frac{s^3}{s^3-\alpha s+1},\infty)=0.
\]
On the other hand, for $k=-1$ we have:
\[
\frac{1}{2\pi i}\int_{|s|=R}\frac{ds}{s(r(s)-\alpha)}=\frac{1}{2\pi i}\int_{|s|=R}\frac{ds}{s^3+1-\alpha s}=0,
\]
here too thanks to the exactity relation, and for $k=-2$ we have:
\[
\begin{split}
  \frac{1}{2\pi i}\int_{|s|=R}\frac{ds}{s^2(r(s)-\alpha)}&=\frac{1}{2\pi i}\int_{|s|=R}\frac{ds}{s(s^3+1-\alpha s)}=0
  \end{split}
\]
and
\[
\sum_{n=1}^3\frac{1}{w_n^2r^\prime(w_n)}+{\rm Res}\, \left(\frac{1}{s(s^3+1-\alpha s)},0\right)=0
\]
and so
\[
  \sum_{n=1}^3\frac{1}{w_n^2r^\prime(w_n)}=-1.
\]
One can also directly compute $X$ for $\alpha=0$. Then the equation $r(z)=\alpha$ becomes $z^3=-1$, and with $j=\frac{1+i\sqrt{3}}{2}$, its roots
  are $w_1=-1,w_2=j$ and $w_3=\overline{j}$. Furthermore $r^\prime(w_n)=2w_n-1/w_n^2=3w_n$. Hence,
  \[
    \sum_{n=1}^3\frac{w_n^k}{r^\prime(w_n)}=\frac{\sum_{n=1}^3w_n^{k-1}}{3}=\begin{cases}\,0,\quad\hspace{2mm} k=-1,0,2\\
      \, 1,\quad\hspace{2mm} k=1, \\
      \, -1,\quad k=-2.
      \end{cases}
   \]

   More generally:

\begin{example}
   The function $r(z)=z^N+\frac{1}{z}$ leads to
   \begin{equation}
     \label{wertyuio}
 X=\begin{pmatrix}& &&0\\
   &\mathcal I_N& &\vdots\\
   & &&0\\
   0& \cdots &0&-1\\
   \end{pmatrix}
  \end{equation}
where $\mathcal I_N$ is the $N\times N$ matrix with antidiagonal elements all equal to $1$, and other elements equal to $0$.
\end{example}

Indeed, we now have $r^\prime(z)=Nz^{N-1}-\frac{1}{z^2}$, and is equal to $-\frac{N+1}{w_n^2}$ when $z=w_n$ is a root of $r(z)=0$, i.e. satisfies
$w_n^{N+1}+1=0$. We then have
\[
  \begin{split}
  X&=\sum_{z\,:\, z^{N+1}+1=0}-z^2\frac{\begin{pmatrix}1\\z\\ \vdots\\z^{N-1}\\ \frac{1}{z}\end{pmatrix}
    \begin{pmatrix}1&z&\cdots&z^{N-1}& \frac{1}{z}\end{pmatrix}}{N+1}\\
  &=\sum_{z\,:\, z^{N+1}+1=0}\frac{-\begin{pmatrix}
z^2&z^3&\cdots&\cdots&z^{N+1}& z\\
z^3&z^4&\cdots&z^{N+1}&z^{N+2}& z^2\\
      & &\iddots && &\\
      & \iddots&\cdots && &\\
      z^{N+1}&z^{N+2}&\cdots&&&\\
      z&z^2&\cdots &&z^N&1
    \end{pmatrix}}{N+1},
  \end{split}
\]
from which \eqref{wertyuio} follows.

\begin{lemma}
  $X(J,r)$ is uniquely defined up to a real similarity matrix $X(J,r)\mapsto SX(J,r)S^t$.
\end{lemma}

\begin{proof}
  Let $Z_r^{(1)}$ be constructed as $Z_r$ but with a different basis. We can write
  \[
Z_r^{(1)}(z)=Z_r(z)S
\]
where $S$ is an invertible matrix; $S$ has real entries since we choose a basis of $\mathfrak L(r)$
with real coefficients. The result follows from
\[
  \sum_{n=1}^N\frac{Z_r^{(1)}(w_n)^tJZ_r^{(1)}(w_n)}{r^\prime(w_n)}=  \sum_{n=1}^N\frac{S^tZ_r(w_n)^tJZ_r(w_n)S}{r^\prime(w_n)}.
\]
\end{proof}

\section{A representation theorem for analytic function}
\setcounter{equation}{0}
Corollary \ref{corocoro} leads to the following natural question: {\sl Which functions can be written in the form $Z_r(z)F(r(z))$?} When $r$ is a polynomial the answer was given in
\cite[Theorem 2.1, p.44]{AJLV15} and in \cite{ajlm1} when $r$ is a finite Blaschke product.
We present a result which encompasses both cases. The result is valid in particular for rational
functions $r$ with possibly less than $N$ pairwise different zeros, but having limit infinity at infinity, and also for Blaschke products, corresponding to
results in \cite{AJLV15} and \cite{ajlm1} respectively. These cases are considered after the proof of the lemma. Note that condition \eqref{lacondition} will not hold for functions $r$
such that $\lim_{z\rightarrow\infty}r(z)=0$.

\begin{theorem} Let $r$ be a rational function with non constant numerator, and let $D_1,\ldots, D_L$ be pairwise disjoint open disks around the distinct zeros $w_1,\ldots, w_L$ of $r$, not containing poles of $r$,
  and such that
$\bigcup_{\ell=1}^L\partial D_\ell$ does not contain poles of $r$ either. We let
\begin{equation}
  \label{rho1}
    \rho=\min_{s\in\bigcup_{\ell=1}^L\partial D_\ell}|r(s)|.
\end{equation}
We assume
\begin{equation}
\Omega_0\stackrel{\rm def.}{=}  \left\{z\in\mathbb C\,:\,|r(z)|<\rho\right\}\subset\bigcup_{\ell=1}^L D_\ell.
  \label{lacondition}
  \end{equation}
Then, every function $f$ analytic in a neighborhood of $\bigcup_{\ell=1}^L \overline{D_\ell}$  can be written in  unique way as
\begin{equation}
    \label{le123211}
    f(z)=Z_r(z)F(r(z)),\quad z\in \Omega_0,
  \end{equation}
  where $F$ is $\mathbb C^N$-valued and analytic in $|z|<\rho$.
  \label{234432}
\end{theorem}

\begin{proof} We proceed in a number of steps.\\

STEP 1: {\sl Let $z\in\bigcup_{\ell=1}^L D_\ell$. It holds that
\[
  f(z)=\frac{1}{2\pi i}
  \sum_{\ell=1}^L\int_{\partial D_\ell}\frac{f(s)}{s-z}ds.
\]
}

This is Cauchy's formula because of the analyticity assumption on $f$.\\

STEP 2: {\sl The function
  \begin{equation}
    \label{Fz}
F(z)=\frac{1}{2\pi i}\sum_{\ell=1}^L\int_{\partial D_\ell}\frac{(I_N-sT)^{-1}bf(s)}{r(s)-z}ds
\end{equation}
satisfies \eqref{le123211} and is analytic in $|z|<\rho$.}\smallskip

Indeed, using \eqref{bghjkl} and \eqref{Zr6789}, we can write
\[
  \begin{split}
    f(z)&=\frac{1}{2\pi i}\sum_{\ell=1}^L\int_{\partial D_\ell}\frac{f(s)}{r(s)-r(z)}\frac{r(s)-r(z)}{s-z}ds\\
    &=\frac{1}{2\pi i}G(I_N-zT)^{-1}\left(\sum_{\ell=1}^L\int_{\partial D_\ell}\frac{(I_N-sT)^{-1}bf(s)}{r(s)-r(z)}ds\right)\\
    &=G(I_N-zT)^{-1}F(r(z))
      \end{split}
    \]
    with $F(z)$ as in \eqref{Fz}. We now check that $F$ is analytic in $|z|<\rho$.
    This follows from the fact that on $\bigcup_{\ell=1}^L\partial D_\ell$ we have $|r(z)|\ge \rho$ by \eqref{lacondition}, and that $(I_N-sT)^{-1}$ is well defined on
    $\bigcup_{\ell=1}^L\partial D_\ell$ since the realization of $r$ is minimal; see Proposition \ref{minizero}.\\

    STEP 3: {\sl The representation \eqref{le123211} is unique.}\smallskip

    This follows from Corollary \ref{corocoro}.
    \end{proof}

    \begin{remark}{\rm
        $F$ is in fact analytic in the set
        \begin{equation}
          \label{domain123}
A(r)=\left\{z\in\mathbb C\,;\, r(z)\not\in\bigcup_{\ell=1}^Lr(\partial D_\ell)\right\}
          \end{equation}
and in particular at infinity. This fact will be used in particular in the sequel in the proofs of various structure theorems.}
    \end{remark}

    \begin{remark}
      {\rm The space $A(r)$ need not be connected, and
        depends on the choice of the disks, $D_1,\ldots, D_\ell$ and so is not unique. In the various structure theorems in the sequel, we fix each time a set $A(r)$. The function
      $f$ is then extended analytically in the (possibly not-connected) set $r^{-1}(A(r))$.}
\label{key}
\end{remark}

\begin{lemma}
\label{lemma123}
  Condition \eqref{lacondition} holds for rational functions regular at the origin and having a pole at infinity.
          \end{lemma}

\begin{proof} We consider open disks $D_{L+1},\ldots, D_M$ around the poles of $r$ (if any), all the disks
  $D_1,\ldots, D_M$ are assumed pairwise non intersecting. As above, we set
   \begin{equation}
    \rho=\min_{s\in\bigcup_{\ell=1}^L\partial D_\ell}|r(s)|.
\end{equation}

STEP 1: {\sl We can choose $D_{L+1},\ldots, D_M$ such that
  \begin{equation}
    \label{36}
\rho<\min_{s\in\bigcup_{\ell=L+1}^M\overline{D_\ell}}|r(s)|
\end{equation}
and in particular
\begin{equation}
  \label{37}
\rho<\min_{s\in\bigcup_{\ell=L+1}^M\partial D_\ell}|r(s)|.
\end{equation}
}

Indeed, write $r(z)=\frac{h_{L+1}(z)}{(z-w_{L+1})^{m_{L+1}}}$, where $m_{L+1}\in \mathbb N$, and  $h_{L+1}(z)$ regular at $w_{L+1}$ with $h(w_{L+1})\not=0$. We can write
\[
h_{L+1}(z)=h(w_{L+1})(1+t_{L+1}(z)),\quad{\rm with}\quad\lim_{z\rightarrow w_{L+1}}t_{L+1}(z)=0.
\]
Let $\epsilon_0$ be such that
\[
|z-w_{L+1}|<\epsilon_0\quad\Longrightarrow\quad |t_{L+1}(z)|<1/2.
\]
Then, for $\epsilon\le \epsilon_0$ we have
\[ s\in \partial B(w_{L+1},\epsilon)\quad\Longrightarrow\quad |r(s)|\ge \frac{1-|t_{L+1}(s)|}{|s-w_{L+1}|^{m_{L+1}}}\ge \frac{1}{2\epsilon^{m_{L+1}}}\rightarrow\infty
\]
as $\epsilon\rightarrow 0$, an so $\lim_{\epsilon\rightarrow 0}\min_{|s-w_{L+1}|=\epsilon}|r(s)|=\infty$.\\

\eqref{37} follows by going over the above construction around each of the poles, and taking the value of $\epsilon$ for which (for instance)
\[
\min_{|s-w_{\ell}|=\epsilon}|r(s)|=2\rho,\quad \ell=L+1,\ldots, M.
\]
We then get \eqref{36} by taking $\epsilon\rightarrow 0$.\\

STEP 2: {\sl There exists $R_0>0$ such that all the closed disks $\overline{D_1},\ldots, \overline{D_M}$ are inside the open disk of radius $R$ for $R\ge R_0$  and
  \begin{equation}
    \label{R0R1}
  \min \left\{|r(z)|:|z|=R\right\}
  \ge2\rho,\quad \forall R\ge R_0.
\end{equation}
}
This is clear since $\lim_{z\rightarrow\infty}|r(z)|=\infty$ implies $\lim_{z\rightarrow\infty}\min \left\{|r(z)|:|z|=R\right\}=\infty$. To check the latter, write
$r(z)=\frac{p_1(z)}{p_2(z)}$, where $p_1$ and $p_2$ are polynomials, of degrees respectively $n_1$ and $n_2$, with $n_1>n_2$ since $r$ has a pole at infinity.
We first assume that $p_2$ is not a constant. Then, with $c_1,c_2\in\mathbb C\setminus\left\{0\right\}$,
\[
    r(z)=\frac{c_1z^{n_1}(1+t_1(z))}{c_2z^{n_2}(1+t_2(z))}    =\frac{c_1}{c_2}z^{n_1-n_2}\frac{1+t_1(z)}{1+t_2(z)}.
\]
To conclude take $R_1$ such that for every $R\ge R_1$,
\[
  |z|>R\quad\Longrightarrow\quad |t_1(z)|<1/2\quad{\rm and}\quad |t_2(z)|<1.
\]
For such $R$,
\[
|1+t_1(z)|\ge 1-|t_1(z)|\ge 1/2\quad{\rm  and}\quad |1+t_2(z)|\le 2,
\]
and so
\[
  \left|\frac{1+t_1(z)}{1+t_2(z)}\right|\ge \frac{1-|t_1(z)|}{1+|t_2(z)|}\ge\frac{1}{4}.
\]
Hence, for $|z|\ge R\ge R_1$,
\[
|r(z)|\ge \frac{R^{n_1-n_2}}{4}\rightarrow \infty\quad {\rm as}\quad z\rightarrow\infty.
\]

By taking $R_1$ large enough, we can therefore find $R_0$ such that \eqref{R0R1} holds when $p_2$ is not a constant polynomial. The case where $p_2$ is a constant works
the same, with considering only the term $1+t_1(z)$.\\

STEP 3: {\sl Let $K=\left\{|z|\leq R\right\}\setminus\bigcup_{\ell=1}^M D_\ell$. We have $|r(z)|\ge \rho$ for $z\in K$.}\smallskip

Indeed, the function $1/r(z)$ is analytic in the interior of the compact set $K$ and continuous on its boundary,
and so it attains its maximum. By the maximum modulus principle and by the definition of $\rho$, this maximum is equal to $1/\rho$.
Hence, in the interior of $K$ we have $1/|r(z)|<1/\rho$, that is $|r(z)|>\rho$.\\

STEP 4: {\sl We have $|r(z)|\ge \rho$ outside the $\bigcup_{\ell=1}^L D_\ell$, and \eqref{lacondition} holds.}\smallskip

That $|r(z)|\ge \rho$ outside the $\bigcup_{\ell=1}^L D_\ell$  follows from the definition of $R_0$ for $|z|>R$, and from \eqref{36} for $z$ inside the closed
disks around the poles.\\
\end{proof}

\begin{corollary}
  \label{c123}
In the notation of Theorem \ref{234432}, one can assume the function $F$ analytic in a neighborhood of any preassigned point $w_0$ different from a pole of $r$.
\end{corollary}

  \begin{proof}
        Let $s_1,\ldots , s_u$ be the roots (if any) of the equation $r(z)=w_0$. In the proof of
        Theorem \ref{234432} we choose the disks $D_1,\ldots, D_M$ so that $s_1,\ldots, s_u\not\in\cup_{\ell=1}^M\overline{D_\ell}$. Then the various steps in the proof of the
        lemma do not change, neither in statement nor in proof.
To conclude we note that $r(\cup_{\ell=1}^M\overline{D_\ell})$ is closed, and since $w_0\not\in r(\cup_{\ell=1}^M\overline{D_\ell})$, there is a neighborhood of $w_0$ not
intersecting $r(\cup_{\ell=1}^M\overline{D_\ell})$. The analyticity claim follows from the formula for $F$.

\end{proof}

\begin{definition}
  \label{def567}
We will denote by $A(r,w_0)$ a set containing $A(r)$ constructed as in the previous corollary.
\end{definition}
We now consider the case of a finite Blaschke product.

\begin{lemma}
  \label{lemmabla}
  Finite Blaschke products satisfy \eqref{lacondition}.
\end{lemma}

\begin{proof}
We consider $D_1,\ldots, D_L$ open disks around the zeros of $r$, and assume them inside the open unit disk and non-intersecting,
and define $\rho$ as in\eqref{rho1}. Then $0<\rho<1$ and inside $B(0,1)\setminus\bigcup_{\ell=1}^L D_\ell$ we have $|r(z)|\ge \rho$ by the maximum modulus principle since $|r(z)|=1$ on the
unit circle. The inequality $|r(z)|\ge \rho$ then still holds outside the unit disk since $r$ is a Blaschke product and $|r(z)|>1$ outside the closed unit disk.
So \eqref{lacondition} also holds for finite Blaschke products.
\end{proof}

We will make use of the previous two results for $w_0=1$ in particular in the proofs of Theorem \ref{tm72} and Theorem \ref{tm102}.\\

The case of vector-valued functions is presented in the next corollary.

\begin{corollary}
  Assume $f$ $\mathbb C^p$-valued. Then representation \eqref{le123211} becomes
  \begin{equation}
    f(z)=(Z_r(z)\otimes I_p)F(r(z))
    \label{le123211111}
    \end{equation}
    where now $F$ is $\mathbb C^{pN}$-valued.
    \end{corollary}
\begin{proof}
The proof follows the one of Theorem \ref{234432} with a modification in Step 2: we define the function $F=\begin{pmatrix}F_1&\cdots&F_p\end{pmatrix})^t$ by
  \begin{equation}
    \label{Fz1}
F_\ell(z)=\frac{1}{2\pi i}\sum_{\ell=1}^L\int_{\partial D_\ell}\frac{(I_N-sT)^{-1}b}{r(s)-z}\, f_\ell(s)ds,
\end{equation}
where $f=\begin{pmatrix}f_1&\cdots& f_p\end{pmatrix}^t$.\smallskip

In fact, reasoning as above to prove Theorem \ref{234432}, we have
\[
  \begin{split}
    f(z)&=\begin{pmatrix}f_1\\ \vdots \\f_p\end{pmatrix}=\frac{1}{2\pi i}\sum_{\ell=1}^L\int_{\partial D_\ell}\frac{f(s)}{r(s)-r(z)}\frac{r(s)-r(z)}{s-z}ds\\
    &=\frac{1}{2\pi i}
    \begin{pmatrix}G(I_N-zT)^{-1}&0&\cdots&0\\
      0&G(I_N-zT)^{-1}&\cdots&0\\
      \vdots&\cdots&\ddots&\vdots \\
      0&\dots&&G(I_N-zT)^{-1}
      \end{pmatrix}
    \left(\sum_{\ell=1}^L\int_{\partial D_\ell}\frac{(I_N-sT)^{-1}b}{r(s)-r(z)}f(s)ds\right)\\
    &=(G(I_N-zT)^{-1}\otimes I_p)F(r(z))
      \end{split}
    \]
    where $F(z)$ is $\mathbb C^{pN}$-valued, by its definition, and is analytic in $|z|<\rho$.
\end{proof}

These representation results allow to make connections with some non-standard interpolation problem, as we now explain.

\begin{example}
  With $r$ as in Example \ref{zz-1} we have
  \begin{equation}
    \label{zz}
    f(z)=\begin{pmatrix}1&\frac{1}{z}\end{pmatrix}F\left(z+\frac{1}{z}\right).
  \end{equation}
\end{example}

\begin{corollary}
Let $f$ be analytic in a domain symmetric with respect to the unit circle and assume $f(z)=f(1/z)$. Then $f(z)=F(z+1/z)$.
\end{corollary}

\begin{proof}
  We write
  \[
    f(z)=\begin{pmatrix}1&\frac{1}{z}\end{pmatrix}F\left(z+\frac{1}{z}\right)
  \]
  and
  \[
    f(1/z)=\begin{pmatrix}1&z\end{pmatrix}F\left(z+\frac{1}{z}\right)
  \]
  It follows that with $F=\begin{pmatrix}F_1\\ F_2\end{pmatrix}$ we have $F_2\equiv 0$ and hence the result.
\end{proof}

Considering the power series expansion of $f$ centered at the origin, a more direct proof would be to prove first by induction that there exists a sequence of polynomials
$P_1,P_2,\ldots$ such that
\[
z^n+\frac{1}{z^n}=P_n\left(z+\frac{1}{z}\right),\quad n=1,2,\ldots
\]

  Representation \eqref{zz} should allow to solve interpolation problem with the interpolation data
  invariant under the map $z\mapsto 1/z$. As a sample problem consider the following:

  \begin{problem}
    Given $w_1,\ldots, w_N$ be $N$ points in the complex plane, and two sets of complex numbers $\{a_1,\ldots, a_N\}$ and $\{b_1,\ldots, b_N\}$, find the set of functions $f$ analytic in neighborhoods of
    the points $w_i$ and $1/w_i$, $i=1,\ldots, N$ and such that
  \[
    \begin{split}
      f(w_i)&=a_i,\\
      f(1/w_i)&=b_i,\quad i=1,\ldots, N.
    \end{split}
  \]
    \end{problem}

    Representation \eqref{le123211} allows to transform this problem into the following tangential interpolation problem with tangential constraints
      \[
    \begin{split}
\begin{pmatrix}1&w_i\end{pmatrix}      F(W_i)&=a_i,\\
\begin{pmatrix}1&1/w_i\end{pmatrix}      F(W_i)&=b_i\quad i=1,\ldots, N.
\end{split}
\]
where $W_i=w_i+\frac{1}{w_i}$. We will not pursue this direction of research in the present work. The next example is related to multipoint interpolation; see \cite{ajlm1,MR3922295}.
More generally, for given $w_1,\ldots, w_N$ such that $r(w_n)=\alpha$, $n=1,\ldots, N$, the above strategy suggests how to solve multipoint interpolation problems; the case
where $r$ is a Blaschke product was considered in \cite{ajlm1}.

\begin{example} Given $w_1,\ldots, w_N,c_1,\ldots, c_N,\gamma\in\mathbb C$, describe the set of functions analytic in a neighborhood of $w_1,\ldots, w_N$ and such that
  \begin{equation}
    \label{multi}
    \sum_{n=1}^N c_nf(w_n)=\gamma.
    \end{equation}
  \end{example}

  Using the representation \eqref{le123211} we rewrite \eqref{multi} as
    \[
      \begin{split}
        \sum_{n=1}^Nc_nf(w_n)&=\gamma\\
        &\iff\\
        \sum_{n=1}^Nc_nZ_r(w_n)F(0)&=\gamma\\
        &\iff\\
        cF(0)&=\gamma,\quad c=\sum_{n=1}^Nc_nZ_r(w_n),
        \end{split}
      \]
      which is a one-point tangential interpolation problem with solutions $F\in\mathfrak H(k)$ (if one can multiply on the left and stay in $\mathfrak H(k)$) given by
      \begin{equation}
        \label{decomp}
F(z)=k(z,0)k(0,0)^{-1}\frac{c^*}{cc^*}\gamma+\left(I-\frac{c^*c}{cc^*}\right)G(z),\quad G\in \mathfrak H(k)
\end{equation}
so that
\[
f(z)=Z_r(z)\left(k(z,0)k(0,0)^{-1}\frac{c^*}{cc^*}\gamma+\left(I-\frac{c^*c}{cc^*}\right)G(z)\right).
\]
Note that the decomposition \eqref{decomp} is orthogonal. Let $d\in\mathbb C^N$. We have:
\[
  \begin{split}
    \langle \left(I-\frac{c^*c}{cc^*}\right)G(z),k(z,0)k(0,0)^{-1}\frac{c^*}{cc^*}d\gamma   \rangle_{\mathfrak H(k)}    &=\overline{\gamma}    d^*ck(0,0)^{-1}k(0,0)\left(I-\frac{c^*c}{cc^*}\right)G(0)\\
    &=\overline{\gamma}d^*c\left(I-\frac{c^*c}{cc^*}\right)G(0)\\
    &=0
  \end{split}
  \]
  since $c\left(I-\frac{c^*c}{cc^*}\right)=0$.\\

We note that such multipoint constraints are considered in numerical analysis to fix degrees of freedom; see \cite[p. 1494]{poggiogirosi}.\\

\begin{theorem}\label{thm4.8}
Let $\mathfrak H(\mathbb K)$ be a reproducing kernel Hilbert space of  $\mathbb C^p$-valued functions of the form \eqref{le123211111}, with reproducing kernel $\mathbb K$.
Assume that the corresponding space of functions $F$, with norm
\begin{equation}
  \label{defnorm}
  \|F\|=\|f\|
\end{equation}
is a \rkhs $\mathfrak H(K)$ of functions analytic in $B(0,\e)$ for some $\e>0$ of the origin. Then $\mathbb K$ and  $K$ are related by
  \begin{equation}
    \label{ker567}
\mathbb K(z,w)=(Z_r(z)\otimes I_p)K(r(z),r(w))(Z_r(w)\otimes I_p)^*.
\end{equation}
\end{theorem}

\begin{proof}
  Consider the space of functions of the form $f(z)=(Z_r(z)\otimes I_p)F(r(z))$, where $F\in\mathfrak H(K)$. By Corollary \ref{corocoro} with $\alpha=0$, we have
\[
  f(z)\equiv 0\quad\Longrightarrow  F(z)\equiv 0,\quad  z\in B(0,\e).
\]
It follows from Theorem \ref{unique123321} that the definition \eqref{defnorm} makes sense and the reproducing kernel is then as in \eqref{ker567}.
\end{proof}

We conclude this section with the following remark and a proposition. It may be of interest to look at functions $r$ of the form $r(z)=(r_0(z))^m$, where $r_0$ is itself rational and $m\in\mathbb N$.
The following proposition connects the corresponding state spaces $\mathfrak L(r)$ and $\mathfrak L(r_0)$.

\begin{proposition}
  Let $r$ be a rational function and let $r_0(z)=(r(z))^m$ with $m\in\mathbb N$. Let $\mathfrak M(r_0)$ to be the linear span of the functions $1,\ldots, r_0^{m-1}$.
  \begin{equation}
\mathfrak L(r)=(\mathfrak M(r_0))\otimes \mathfrak L(r_0),
   \end{equation}
   meaning that $t\in\mathfrak L(r)$ if and only if it can be written as
   \begin{equation}
t(z)=\sum_{u=0}^{m-1}t_u(z)r_0^u(z),\quad t_0,\ldots, t_{m-1}\in\mathfrak (r_0).
     \end{equation}
\end{proposition}

\begin{proof}
  The result follows from
  \[
\frac{r_0^m(z)-r_0^m(a)}{z-a}=\sum_{u=0}^{m-1}r_0^u(z)\frac{r_0(z)-r_0(a)}{z-a},
\]
where $z,a$ are points of analyticity of $r$.
\end{proof}

We conclude this section with a formula:

\begin{proposition}
  Let $F$ be given by \eqref{Fz}, extended to the domain $A(r)$ (see \eqref{domain123}). Then, for every $\alpha\in A(r)$
  \begin{equation}
R_\alpha F(z)=\frac{1}{2\pi i}\sum_{\ell=1}^L\int_{\partial D_\ell}\frac{(I_N-sT)^{-1}bf(s)}{(r(s)-\alpha)(r(s)-z)}ds .
    \end{equation}
  \end{proposition}

  \begin{proof}
    This follows from
    \[
      \frac{1}{r(s)-z)}-\frac{1}{r(s)-\alpha}=\frac{z-\alpha}{(r(s)-\alpha)(r(s)-z)}.
      \]
    \end{proof}
\section{Cuntz relations}
\setcounter{equation}{0}
\label{sec2021}
Representations of the Cuntz algebra (see \cite{Cun77}) play an important role in analysis, wavelets and signal processing, in particular in the theory of filter banks; see \cite{BrJo97,MR2254502,MR2277210}.
Connections between the Cuntz relations and reproducing kernel spaces were
considered in \cite{ajlm2,ajlm1}. In the present section we show how the Cuntz relations
are related to the present setting.\smallskip

Let in \eqref{le123211111}
\[
F=\begin{pmatrix}F_1\\ \vdots\\ F_N\end{pmatrix}
\]
where $F_1,\ldots, F_N$ are $\mathbb C^p$-valued. We thus have
\begin{equation}
  \label{decompo}
  f(z)=\sum_{n=1}^N e_n(z)F_n(r(z)).
\end{equation}
\begin{theorem}
  If $K$ in \eqref{ker567} is block diagonal,
  \[
    K={\rm diag}\, (k_1,\ldots, k_N)
  \]
with possibly different positive definite $\mathbb C^{p\times p}$-valued kernels $k_1,\ldots, k_N$, the formulas
\[
  (T_nf)(z)=F_n(z),\quad n=1,\ldots, N,
\]
define bounded operators $T_n\in\mathbf L(\mathfrak H(\mathbb K),\mathfrak H(k_n))$, $n=1,\ldots, N$,
with adjoint operators
\[
(T_n^*g)(z)=e_n(z)g(r(z))
\]
and $T_1^*,\ldots, T_N^*$ satisfy the Cuntz relations.
\label{tm2021}
\end{theorem}

\begin{proof} The operators $T_1,\ldots, T_N$ are well defined in view of Theorem \ref{unique123321}. They are bounded by definition of the norms in \eqref{defnorm} since
  \[
\|F\|^2_{\mathfrak H(K)}=\sum_{n=1}^N\|F_N\|^2_{\mathfrak H(k_n)}.
   \]
   To compute the adjoint operators we note that
   \begin{equation}
T_n(\mathbb K(\cdot, w)d) =k_n(\cdot, r(w))\overline{e_n(w)}d_n,\quad{\rm with}\quad d=\begin{pmatrix}d_1\\ \vdots\\ d_N\end{pmatrix}, \quad d_n\in\mathbb C^p.
\end{equation}
Then it follows from a classical result on weighted composition operators that
\[
(T_n^*g)(z)=e_n(z)g(r(z)),\quad n=1,\ldots, N.
\]
Furthermore we have from \eqref{decompo} that
\[
\sum_{n=1}^NT_n^*T_n=I_{\mathfrak H(\mathbb K)}.
\]
Finally for $n,m\in\left\{1,\ldots, N\right\}$, we have
\[
  \begin{split}
    (T_nT_m^*g)(z)&=(T_n(e_mg\circ r))(z)\\
    &=\begin{cases}\, g(z),\,\, n=m\\
      \, 0,\,\,\,\,\hspace{5mm} n\not =m,\end{cases}
  \end{split}
\]
since the function $e_m(z)g(r(z))$ has unique decomposition along the direct sum corresponding to \eqref{decompo}.
\end{proof}

In the case where $k_1=\cdots= k_N=k$ it is of interest to see when it holds that
\[
  k(z,w)=(Z_r(z)\otimes I_p)k(r(z),r(w))(Z_r(w)\otimes I_p)^*,
\]
i.e.
\begin{equation}
  \label{itzik}
k(z,w)=(\sum_{n=1}^N e_n(z)\overline{e_n(w)})k(r(z),r(w)).
\end{equation}
Such equations, and their connections with the Cuntz relations,
have been studied in \cite{ajlm2}.

\section{$R_\alpha^{(r)}$-invariant subspaces}\setcounter{equation}{0}\label{Sec6}
\setcounter{equation}{0}
Recall that $Z_r(z)$ has been defined by \eqref{Zr6789}. We first consider the finite dimensional case.

\begin{theorem}
  A finite dimensional vector space $\mathfrak M$ of $\mathbb C^m$-valued functions is $R^{(r)}_\alpha$-invariant if and only if it is spanned by the columns of a matrix of the form
  \begin{equation}
    \label{345543}
    \M(z)=(Z_r(z)\otimes I_m)C(A-r(z)B)^{-1}
  \end{equation}
  where $C\in\mathbb C^{mN\times M}$ and $A,B\in\mathbb C^{M\times M}$ satisfying $\det(A+\alpha B)\not =0$ for some integer $M$. Furthermore, $M$ can be chosen to be the dimension of $\mathfrak M$, and
  we can always assume that the matrices $A$ and $B$ satisfy $A+\alpha B=I_M$.\smallskip

  For a fixed choice of $Z_r(z)$, the representation \eqref{345543} is unique up to similarity matrix $S$, meaning
  \begin{eqnarray}
    C_1&=&CS\\
    A_1&=&S^{-1}AS\\
    B_1&=&S^{-1}BS.
           \end{eqnarray}
\end{theorem}

\begin{proof}
Let $M$ denote the dimension of $\mathfrak M$, and let $\M(z)$ be a $m\times M$-valued function with columns made from a basis of $\mathfrak M$. There exists $T\in\mathbb C^{M\times M}$ such that
\[
R^{(r)}_\alpha \M=\M T
\]
that is,
\[
\frac{\M(z)}{r(z)-\alpha}-\sum_{n=1}^N\frac{\M(w_n)}{r^\prime(w_n)(z-w_n)}=G(z)T,
\]
and so
\begin{equation}
  \M(z)(I_M-(r(z)-\alpha)T)=\sum_{n=1}^N\frac{\M(w_n)}{r^\prime(w_n)}\frac{r(z)-\alpha}{z-w_n}.
\end{equation}
The functions
\[
\frac{r(z)-\alpha}{z-w_n}=\frac{r(z)-r(w_n)}{z-w_n}
\]
belong to $\mathfrak L(r)$. So with a preassigned basis $e_1,\ldots, e_N$ and $Z_r$ as in
\eqref{Zr6789} we can write
\begin{equation}
  \M(z)=(Z_r(z)\otimes I_m)C(I_M-(r(z)-\alpha)T)^{-1}
\end{equation}
where $C\in\mathbb C^{mN\times M}$.\smallskip

Conversely, let $\M$ be of the form \eqref{345543}, where $C\in\mathbb C^{mN\times M}$ and $A,B\in\mathbb C^{M\times M}$ satisfying $\det(A+\alpha B)\not =0$. We can write:
\[
  \begin{split}
    (R^{(r)}_\alpha \M)(z)&=\frac{(Z_r(z)\otimes I_m)C(A-r(z)B)^{-1}}{r(z)-\alpha}-\sum_{n=1}^N\frac{(Z_r(w_n)\otimes I_m)C(A-r(w_n)B)^{-1}}{r^\prime(w_n)(z-w_n)}\\
   & \\
&=\frac{(Z_r(z)\otimes I_m)C(A-r(z)B)^{-1}}{r(z)-\alpha}-\left(\sum_{n=1}^N\frac{(Z_r(w_n)\otimes I_m)}{r^\prime(w_n)(z-w_n)}\right)C(A-\alpha B)^{-1}\\
&\\
&=(Z_r(z)\otimes I_m)\frac{C(A-r(z)B)^{-1}-C(A-\alpha B)^{-1}}{r(z)-\alpha}\\
&\\
&=(Z_r(z)\otimes I_m)C(A-r(z)B)^{-1}B(A-\alpha B)^{-1}\\
&\\
&=\M(z)B(A-\alpha B)^{-1},
\end{split}
\]
where we have used \eqref{2.5} to go from the second line to the third one.\smallskip

One can choose $M={\rm dim}\,\mathfrak M$ by taking the columns of $M(z)$ linearly independent. Finally, the last remark follows from replacing $C,A$ and $B$ by $C(A+\alpha B)^{-1}$,
$A(A+\alpha B)^{-1}$ and $B(A+\alpha B)^{-1}$ respectively.\smallskip

Assume that the space has two representations as the span of
\[
  (Z_r(z)\otimes I_m)C(A-r(z)B)^{-1}B(A-\alpha B)^{-1}\quad {\rm and}\quad (Z_r(z)\otimes I_m)C_1(A_1-r(z)B_1)^{-1},
\]
with the properties that $C,C_1\in\mathbb C^{m\times N}$, $A,B,A_1,B_1\in\mathbb C^{N\times N}$, and such that $A+\alpha B=A_1+\alpha B_1=I_M$.
There exists an invertible matrix $S$ such that
\[
(Z_r(z)\otimes I_m)C(A-r(z)B)^{-1}B(A-\alpha B)^{-1}=(Z_r(z)\otimes I_m)C_1(A_1-r(z)B_1)^{-1}S.
\]
Using Corollary \ref{corocoro1} it is sufficient to compare
\[
  C(A-r(z)B)^{-1}B(A-\alpha B)^{-1}=C_1(A_1-r(z)B_1)^{-1}S,
\]
i.e.
\[
  C(I-(e-\alpha)B)^{-1}=C_1(I-(e-\alpha)B_1)^{-1}S
\]
with $e\in\mathbb C$ and  $e-\alpha$ small enough. But this is then  classical result from linear system theory.
\end{proof}

The case where $\mathfrak M$ has dimension $1$ will correspond to eigenfunctions, and we get:
\begin{corollary}
Eigenvectors of the operator $R^{(r)}_\alpha$ are exactly the functions of the form
\begin{equation}
f(z)=\frac{Z_r(z)}{a-r(z)b}
\end{equation}
with corresponding eigenvalue $b$, where $a,b\in\mathbb C$ are such that $a-\alpha b\not=0$.
\end{corollary}

Furthermore:
\begin{corollary}
A finite dimensional $R^{(r)}_\alpha$-invariant reproducing kernel Pontryagin space has a kernel of the form
\begin{equation}
\label{yevozavut}
K(z,w)=(Z_r(z)\otimes I_m)C(A-r(z)B)^{-1}P^{-1}(A-r(w)B)^{-*}C^*(Z_r(w)\otimes I_m)^*,
\end{equation}
for some Hermitian invertible matrix $P$.
\end{corollary}

\begin{proof}
This is a direct consequence of the formula for the reproducing kernel of a finite dimensional space.
\end{proof}

An important special case of \eqref{yevozavut}, where
\[
  C(A-r(z)B)^{-1}P^{-1}(A-r(w)B)^{-*}C^*={\rm diag}\, (K(z,w),K(z,w),\ldots ,K(z,w)).
\]
In the next results we connect properties of $P$ to the structure of the kernel.

\begin{theorem}
  Assume $(C,A,B)$ minimal. Then the kernel \eqref{yevozavut} is of the form
  \begin{equation}
    (Z_r(z)\otimes I_m)\frac{J-\Theta(r(z))\Theta(r(w))^*}{1-r(z)\overline{r(w)}}(Z_r(w)\otimes I_m)^*
      \label{fd!!!}
   \end{equation}
   if and only if $P$ solves the equation
   \begin{equation}
     \label{P??}
     A^*PA-B^*PB=C^*JC .
   \end{equation}
   \label{poiuy}
 \end{theorem}

 \begin{proof}
   This follows from \cite[(5.12)]{ad-jfa} or \cite[Theorem 2.1, p. 37]{MR1246809} with $a(z)=1$ and $b(z)=r(z)$ in the notation of these papers.
\end{proof}

We refer to \cite[Lemma 2.2, p. 415 and (2.4) p. 416]{a-seville} for the following result; a proof is given for completeness.

\begin{theorem}
  \label{thm6-5}
  Let $\mathfrak M$ be a space of finite dimension $d$ of $\mathbb C^p$-valued functions analytic in some open set $\Omega$ and assume that there is a point
  $a_0\in\Omega$ such that the linear span of the
  vectors $f(a_0)$ with $f$ varying in $\mathfrak M$ is equal to $\mathbb C^p$. Then,
$\mathfrak M$ satisfies for every $a\in\Omega$ (at the possible exception of $a$ in a zero set) the condition
\begin{equation}
  \label{invdb}
f\in\mathfrak M,\,\,and\,\,\,    f(a)=0\quad\Longrightarrow \frac{f(z)}{z-a}\in\mathfrak M,
  \end{equation}
if and only if it is spanned by the columns of a matrix of the form
\begin{equation}
  \label{EABC}
    E(z)C(A-zB)^{-1}.
  \end{equation}
  where $E$ is $\mathbb C^{p\times p}$-valued analytic in $\Omega$ and such that $\det E(z)\not\equiv 0$, and where $A,B$ and $C$ are matrices of appropriate size and such that
  $\det (A-zB)\not\equiv 0$. Furthermore, $A$ and $B$ can be chosen such that $A-a_0B=I_d$.
\end{theorem}

\begin{proof} We proceed in a number of steps.\\

  STEP 1: {\sl Construction of the function $E(z)$.}\smallskip

  The space $\mathbb C^p$ is spanned by the vectors $f(a_0)$, and so contains a basis made of such vectors, say $f_1(a_0),\ldots, f_p(a_0)$, where $f_1,\ldots, f_p\in\mathfrak M$. The function
  $E(z)$ with columns $f_1(z),\ldots, f_p(z)$ has a non-identically vanishing determinant since $E(a_0)$ is invertible by construction.\\

  STEP 2: {\sl Construction of a basis for $\mathfrak M$.}\smallskip

  Let $E(z)$ be a $\mathbb C^{p\times p}$-valued function $E(z)$ as in the previous step, and let $d$ denote the dimension of $\mathfrak M$,
  Let $F(z)$ be a  $\mathbb C^{p\times d}$-valued function, whose columns form a basis of $\mathfrak M$. For every $c\in\mathbb C^p$ the function $z\mapsto F(z)c-E(z)E^{-1}(a_0)F(a_0)c$ belongs
    to $\mathfrak M$ and vanishes at $a_0$. By hypothesis \eqref{invdb}, the function
    \[
   z\mapsto \frac{F(z)c-E(z)E^{-1}(a_0)F(a_0)c}{z-a_0}
 \]
 also belongs to $\mathfrak M$. It follows that there exists $G\in\mathbb C^{d\times d}$ such that
    \[
\frac{F(z)-E(z)E(a_0)^{-1}F(a_0)}{z-a_0}=F(z)G,
\]
and hence
\[
F(z)=E(z)E(a_0)^{-1}F(a_0)(I_N-(z-a_0)G)^{-1},
\]
which is of the form \eqref{EABC} with $C=E(a_0)^{-1}F(a_0)$, $A=I_d+a_0G$ and $B=a_0G$.\\

STEP 3: {\sl Converse statement:}\smallskip

Conversely,  let $F(z)$ be of the form \eqref{EABC} and let $a$ be such that $\det E(a)\not=0$. We have for $c\in\mathbb C^p$ that $F(a)c=0$ if and only if $C(A-aB)^{-1}c=0$. We can write
\[
\begin{split}
  \frac{F(z)c}{z-a}  &=\frac{E(z)C\left((A-zB)^{-1}-(A-aB)^{-1}\right)c}{z-a}\\
  &=E(z)C\left((A-zB)^{-1}\left\{\frac{A-aB-A+zB}{z-a}\right\}(A-aB)^{-1}\right)c\\
&=E(z)C(A-zB)^{-1}B(A-aB)^{-1}c
\end{split}
\]
and so the function $z\mapsto \frac{F(z)c}{z-a}\in\mathfrak M$.
\end{proof}

We can now prove the following result, which characterizes the counterpart of classical finite dimensional de Branges spaces in our setting.

\begin{theorem} Let $r$ be a {\rm real} rational function satisfying the hypothesis of Theorem \ref{234432}. Under the notation of the previous theorem,
let $\mathfrak M$ be a space of finite dimension $d$ of $\mathbb C^p$-valued functions analytic in some open set $\Omega$ and assume that there is a point
  $a_0\in\Omega$ such that the linear span of the
  vectors $f(a_0)$ with $f$ varying in $\mathfrak M$. Assume furthermore that
$\mathfrak M$ has the following property: Let $\alpha\in\mathbb C$ for which the equation $r(z)=\alpha$
  has $N$ distinct solutions $w_1,\ldots, w_N$, and such that
  \begin{equation}
    \label{ath}
    f(w_n)=0,\,\,\,\,n=1,\ldots, N\,\,\Longrightarrow\,\, z\mapsto \frac{f(z)}{r(z)-\alpha}\,\in\mathfrak M.
  \end{equation}
  Then $\mathfrak M$
  has reproducing kernel of the form
    \[
      (Z_r(z)\otimes I_p)  \frac{E_+(r(z))JE_+(r(w))^*-E_-(r(z))JE_-(r(w))^*}{-i(r(z)-\overline{r(w)})}(Z_r(w)\otimes I_p)^*,
      \]
    if and only if $P$ satisfies \eqref{P??}.
 \end{theorem}
 \begin{proof}
   We write the elements of the space in the form \eqref{le123211111}, $f(z)= (Z_r(z)\otimes I_p)F(r(z))$, where $F$ is analytic in $|z|<\rho$ where $\rho$ is
   independent of $F\in\mathfrak M$; see Theorem \ref{234432}.
      To proceed,
   we divide the argument in a number of steps.\\

   STEP 1: {\sl
     The space $\mathfrak N$ of functions $F(r(z))$ is finite dimensional.}\smallskip

   By Corollary \ref{corocoro} the function $F$ is uniquely determined by $f$, and so $\mathfrak N$ is finite dimensional of dimension $N$
   since $\mathfrak M$ is of this type.\smallskip

   We denote by
   $M(r(z))$ a matrix-valued functions whose columns are a basis of this space.
   % The function $r$ is invertible in a neighborhood of $\alpha$ and define $M(z)$ such that $M(r(z))=M_1(z)$ in
%   such a neighborhood.
   We define an inner product via
   \[
       [(Z_r(z)\otimes I_p)M(r(z))\xi, (Z_r(z)\otimes I_p)M(r(z))\eta]=\eta^*P\xi,\quad \eta,\xi\in\mathbb C^N,
     \]
     where $P$ is Hermitian and invertible.\\

     STEP 2: {\sl The space of functions $M(r(z))\xi$ satisfies the invariance condition \eqref{invdb}.}\smallskip

    Indeed, by assumption, $z\mapsto (Z_r(z)\otimes I_p)\frac{M(r(z))}{r(z)-\alpha}\in\mathfrak M$ since $M(r(w_n))=M(\alpha)=0$.\\

     and hence $M(z)$ is of the form \eqref{EABC}. Applying Theorem \ref{thm6-5} we see that the reproducing of $\mathfrak M$ is therefore equal to
     \[
(Z_r(z)\otimes I_p)E(r(z))C(A-r(z)B)^{-1}P^{-1}(A-r(w)B)^{-*}C^*(Z_r(w)^*\otimes I_p)
     \]
     and the proof is concluded by using Theorem \ref{poiuy}.
  \end{proof}

We now turn to the infinite dimensional case, and consider a Hilbert space $\mathfrak H$ of $\mathbb C^m$-valued function analytic in some open set $\Omega$, and
$R^{(r)}_\alpha$-invariant. We assume $\mathfrak H$ separable. We choose an orthonormal basis $b_1(z),...$ of $\mathfrak H$ and consider the unitary map $T$
\[
T(c_1,...) = \sum_{j=1}^\infty c_j b_j
\]
from $\ell^2(\mathbb N,\mathbb C^m)$ onto $\mathfrak H$. Since the space $\mathfrak H$ is assumed $R^{(r)}_\alpha$-invariant, for every
$c\in \ell^2(\mathbb N,\mathbb C^m)$ we have $R^{(r)}_\alpha T c = T d$ for some uniquely defined $d$ in $\ell^2(\mathbb N,\mathbb C^m)$. The map $c\mapsto d$ is linear
and bounded since
\[
||R^{(r)} T c ||_{\mathfrak H}\leq K ||T c||_{\mathfrak H},
\]
and
\[
||R^{(r)} T c ||_{\mathfrak H} = ||T d||_{\mathfrak H},
\]
implying
\[
||T d ||_{\mathfrak H}\leq K ||T c||_{\mathfrak H}.
\]
We write $d = M c$, where $M:\,l^2 \to l^2$ is bounded since
\[
||T d ||_{\mathfrak H} = ||d||_{\ell^2(\mathbb N,\mathbb C^m)},
\]
and
\[
||T c ||_{\mathfrak H} = ||c||_{\ell^2(\mathbb N,\mathbb C^m)}.
\]
So, we obtain
\[
R^{(r)}_\alpha T\,c = T\,M\, c,\,\,\,\forall c\in \ell^2(\mathbb N,\mathbb C^m).
\]
Hence,
\[
\frac{(T c)(z)}{r(z) - \alpha} - \sum_{j=1}^N \frac{(T c)(w_j)}{r'(w_j)(z-w_j)} = (T\,M\,c)(z).
\]
So,
\[
(T c)(z) - (r(z) - \alpha) (T\,M\,c)(z) = \sum_{j=1}^N \frac{(T c)(w_j)}{r'(w_j)}\frac{r(z) - \alpha}{z-w_j}.
\]
The argument is then the same as in the finite dimensional case.

\section{Operator models}
\setcounter{equation}{0}
In this section we assume that the kernel $K(z,w)$ in \eqref{ker567} is diagonal, in the form
\[
    K={\rm diag}\, (k,\ldots, k)
  \]
  where $k(z,w)$ is $\mathbb C^{p\times p}$-valued and positive definite.
Recall that $r$ is a rational function of degree $N$, with a pole at infinity, and hence of the
form $r(z)=p(z)/q(z)$, with ${\rm deg}\, p=N$ and ${\rm deg }\, q<N$.

\begin{proposition} Let $\alpha\in\Omega(r)$ and consider $R^{(r)}_\alpha$ with domain of definition the space function analytic in neighborhoods of the points $w_1,\ldots, w_N$. Then
it holds that
\begin{equation}
\ker  R^{(r)}_\alpha=\mathfrak L(r)\otimes \mathbb C^p
  \end{equation}
\end{proposition}

\begin{proof} We assume $p=1$ to lighten the notation. Assume that $f\in\ker R^{(r)}_\alpha$. Then,
  \[
f(z)=\sum_{n=1}^N\frac{f(w_n)}{r^\prime(w_n)}\frac{r(z)-r(w_n)}{z-w_n}\in\mathfrak L(r)
\]
where $r(w_n)=\alpha$ and since the function
\[
z\mapsto \frac{r(z)-r(\beta)}{z-\beta}\in\mathfrak L(r)
\]
for every complex number which is not a pole of $r$. The functions $R_wr$ form a basis of $\mathfrak L(r)$ for any set of $N$ pairwise different points, and so the kernel of
$R^{(r)}_\alpha$ is exactly equal to $\mathfrak L(r)$. More precisely, any function $f$ in $\mathfrak L(r)$ is of the form
\[
f(z)=\sum_{n=1}^N\frac{c_n}{r^\prime(w_n)}\frac{r(z)-r(w_n)}{z-w_n}\in \mathfrak L(r), \quad c_1,\ldots, c_N\in \mathbb C,
\]
and is such that $f(w_m)=c_m$, $m=1,\ldots, N$, as is seen by plugging $z=w_m$ in the above equation. So $f\in\ker R^{(r)}_\alpha$.
\end{proof}

We now use \eqref{formulacoro} to relate operators associated to the operators $R_\alpha$ and $R_\alpha^{(r)}$. The existence of the operators $\mathsf A$ and $\mathsf A_r$ in the next theorem
follow from \cite[Theorem 4.10, p. 137]{Stone}).

\begin{theorem}
Assuming that the operators $R^{(r)}_\alpha$ and $R_\alpha$ are bounded in some Hilbert space $\mathfrak H$, and have trivial kernels in $\mathcal H$ for every $\alpha\in\Omega(r)$. Then,
there exist two possibly unbounded operators $\mathsf A$ and $\mathsf A_r$ such that
\begin{equation}
R^{(r)}_\alpha f=(\mathsf A_r-\alpha I)^{-1}\quad {\it and}\quad R_\alpha F=(\mathsf A-\alpha I)^{-1}.
\end{equation}
They are related by
\[
((\mathsf A_r-\alpha I)^{-1} f)(z)=\sum_{n=1}^Ne_n(z)((\mathsf A-\alpha I)^{-1}F_n)(r(z))
\]
when $f$ is decomposed as \eqref{decompo},
\begin{equation*}
%  \label{decompo}
  f(z)=\sum_{n=1}^N e_n(z)F_n(r(z)).
\end{equation*}
\end{theorem}

\begin{theorem}
In the notation of the previous theorem, assume that
  \begin{equation}
    \label{inter678}
    \mathfrak H\cap \mathfrak L(r)=\left\{0\right\}.
    \end{equation}
  Then
  \begin{equation}
    {\rm Dom}\, \mathsf A_r=\left\{f\in\mathfrak H\,\, \text{for which there exists $h_f\in\mathfrak L(r)$ such that } r(z)f(z)+h_f(z)\in\mathfrak H\right\}
  \end{equation}
  and then
  \begin{equation}
    \label{newmodel}
    (\mathsf A_r f)(z)=r(z)f(z)+h_f(z).
    \end{equation}
\end{theorem}

\begin{proof}
  By \eqref{inter678} we have $\ker R^{(r)}_\alpha=\left\{0\right\}$ for some (and hence all) $\alpha$, and so the operator $\mathsf A_r$ exists; see
\cite[Theorem 4.10, p. 137]{Stone}).
\end{proof}

We remark that, with $(\mathsf A f)(z)=zf(z)+d_f$ we have

\[
  \begin{split}
    \sum_{n=1}^Ne_n(z)((A-\alpha I)^{-1}f_n)(r(z)&=\sum_{n=1}^Ne_n(z)(r(z)f_n(r(z))+d_{f_n})\\
    &=r(z)\left(\sum)_{n=1}^N e_n(z)f_n(r(z))\right)+\sum_{n=1}^Ne_n(z)d_{f_n}\\
    &=r(z)Z_r(z)F(r(z))+h_f(z)\\
    &=r(z)f(z)+h_f(z)
  \end{split}
  \]
with $f(z)=Z_r(z)F(r(z))$ and $h_f(z)=\sum_{n=1}^N e_n(z)d_{f_n}$.\\

\section{$\mathfrak P(\Theta)$ spaces}
\setcounter{equation}{0}
We now study the counterparts of reproducing kernel Pontryagin spaces with reproducing kernel of the form
\begin{equation}
\frac{J-\Theta(z)J\Theta(w)^*}{-i(z-\overline{w})},\quad{\rm or}\quad\frac{J-\Theta(z)J\Theta(w)^*}{1-z\overline{w}},
\end{equation}
in the present setting. The first formula corresponds to the {\sl line case}, while the second corresponds to the {\sl circle case.}
These kernels give models for close to self-adjoint and close to unitary operators in the classical case; see \cite{MR48:904,l1} for the latter.\smallskip

We begin with the line case.
In the statement of the next theorem, $J\in\mathbb R^{p\times p}$ is a signature matrix and $X(J,r)$ is the corresponding associated symmetric matrix (see \eqref{Pgram}).
In the statements we consider functions $r$ which satisfy the conditions of Theorem \ref{234432}. This includes in particular rational functions regular at the origin and having a pole at infinity
(see Lemma \ref{lemma123}) and finite Blaschke products. Recall that $\Omega(r)$ is defined by \eqref{rtyuiop}, that
$A(r)$ is defined in \eqref{domain123} and that we assume functions extended to $r^{-1}(A(r))$; see Remark \ref{key}.

\begin{theorem} Let $r$ be a {\rm real} rational function satisfying the hypothesis of Theorem \ref{234432}.
  Let $\mathfrak P$ be a reproducing kernel Pontryagin space of $\mathbb C^p$-valued functions analytic in $r^{-1}(A(r))$. Then, for each $\alpha,\beta\in \Omega(r)\cap A(r)$
  such that $\alpha\not=\overline{\beta}$, it holds that
\begin{equation}
  \left[ R^{(r)}_\alpha f,g\right]_{\mathfrak P} -\left[ f, R_\beta^{(r)} g\right]_{\mathfrak P}  -(\alpha-\overline{\beta})\left[ R^{(r)}_\alpha f,R^{(r)}_\beta g\right]_{\mathfrak P}
  =i(\alpha-\overline{\beta})\sum_{i,j=1}^N
  \frac{g(v_j)^*}{\overline{r^\prime(v_j)}}J\frac{f(u_i)}{r^\prime(u_i)}\frac{1}{u_i-\overline{v_j}}
\end{equation}
(with $r(u_i)=\alpha$ and $r(v_j)=\beta$, $i,j=1,\ldots, N$)
  if and only if the reproducing kernel of $\mathfrak P$ is of the form
  \begin{equation}
 ( Z_r(z)\otimes I_p)\frac{X(J,r)^{-1}-\Theta(r(z))X(J,r)\Theta(r(w))^*}{-i(r(z)-\overline{r(w)})}(Z_r(w)\otimes I_p)^*
\end{equation}
where $\Theta$ is a $\mathbb C^{Np\times Np}$-valued function analytic in a neighborhood of the origin.
\label{23443211}
  \end{theorem}

\begin{proof}
  Using \eqref{formulacoro} we can write with $f=Z_rF(r)$ and $g=Z_rG(r)$, where $F$ and $G$ are analytic in $A(r)$. By Corollary \ref{corocoro} applied to every connected component of
  $r^{-1}(A(r))$ we see that $F$ and $G$ are uniquely determined, and we can define an inner product by
  \begin{equation}
    \label{inner1}
    [F,G]=[f,g]_{\mathfrak P}.
    \end{equation}

Fo $\alpha,\beta\in A(r)$ we have
\begin{equation}
  \begin{split}
    \left[ R^{(r)}_\alpha f,g\right]_{\mathfrak P} -\left[ f, R_\beta^{(r)} g\right]_{\mathfrak P}  -(\alpha-\overline{\beta})\left[ R^{(r)}_\alpha f,R^{(r)}_\beta g\right]_{\mathfrak P}  &=\\
  &\hspace{-4cm} = \left[ R_\alpha F,G\right] -\left[ F, R_\beta G\right]  -(\alpha-\overline{\beta})\left[ R_\alpha F,R_\beta G\right],\quad
  \end{split}
\end{equation}

Using \eqref{2.5} we can write
\[
\begin{split}
  (\alpha-\overline{\beta})\sum_{i,j=1}^N\frac{g(v_j)^*}{\overline{r^\prime(v_j)}}J\frac{f(u_i)}{r^\prime(u_i)}\frac{1}{u_i-\overline{v_j}}&=
  G(\beta)^* \left(\sum_{i,j=1}^N\frac{Z_r(v_j)^*}{\overline{r^\prime(v_j)}}J\frac{Z_r(u_i)}{r^\prime(u_i)}\frac{\alpha-\overline{\beta}}{u_i-\overline{v_j}}\right)F(\alpha)\\
  &=   G(\beta)^* \left(\sum_{j=1}^N\frac{Z_r(v_j)^*}{\overline{r^\prime(v_j)}}J\left(\sum_{i=1}^N\frac{Z_r(u_i)}{r^\prime(u_i)}\frac{\overline{\beta}-\alpha}{\overline{v_j}-u_i}\right)
\right)F(\alpha)\\
      &=   G(\beta)^* \left(\sum_{j=1}^N\frac{Z_r(v_j)^*}{\overline{r^\prime(v_j)}}J\left(\sum_{i=1}^N\frac{Z_r(u_i)}{r^\prime(u_i)}\frac{\overline{\beta}-\alpha}{\overline{v_j}-u_i}\right)
      \right)F(\alpha)\\
            &=   G(\beta)^* \left(\sum_{j=1}^N\frac{Z_r(v_j)^*}{\overline{r^\prime(v_j)}}J\left(Z_r(\overline{v_j})\frac{\overline{\beta}-\alpha}{\overline{\beta}-\alpha}\right)
            \right)F(\alpha)\\
            &=   G(\beta)^* XF(\alpha),
\end{split}
\]
where $X(J,r)$ is defined by \eqref{Pgram} (and is Hermitian, invertible, and independent of $\alpha$ and $\beta$). We are therefore back to the classical case,

\[
\left[ R_\alpha F,G\right] -\left[ F, R_\beta G\right]  -(\alpha-\overline{\beta})\left[ R_\alpha F,R_\beta G\right]=iG(\beta)^*X(J,r)F(\alpha)
\]
Write $X(J,r)=Y^*J_0Y$, where $J_0$ is a real signature matrix. The space $\mathfrak P_1$ of functions of the form $YF$ with inner product
\begin{equation}
  \label{inner2}
  [YF,YG]_Y=[F,G]
\end{equation}
satisfies
\[
\left[ R_\alpha YF,YG\right] -\left[ YF, R_\beta YG\right]  -(\alpha-\overline{\beta})\left[ R_\alpha YF,R_\beta YG\right]=iG(\beta)^*Y^*J_0YF(\alpha)
\]
and are analytic across the real line (since the functions $F$ are analytic in a disk $|z|<\rho$ by definition of $A(r)$). This analyticity condition allows us to apply the Pontryagin space version of
de Branges' result (see \cite[Theorem 6.10]{ad3}, and see \cite[Theorem III, p. 447]{dbhsaf1}, \cite[Theorem 2.3 p. 598]{ad1} for the Hilbert space version) to assert that its reproducing kernel is of the form
\begin{equation}
  \label{ttttt}
\frac{J_0-\Theta_0(z)J_0\Theta_0(w)^*}{-i(z-\overline{w})}.
  \end{equation}
  from which the result follows. More precisely, we see that the index of the reproducing
  kernel Pontryagin space with reproducing kernel \eqref{ttttt} is equal to $\nu_-(\mathfrak P)$.
\end{proof}

\begin{notation}
  We will set $\mathfrak P=\mathfrak P(\Theta)$ and define
  \begin{equation}
    \nu_-(\Theta)=\nu_-(\mathfrak P),
  \end{equation}
  where the latter is the index of negativity of the Pontryagin space $\mathfrak P$.
\end{notation}

We now consider an inclusion and factorization result related to the above spaces. In the statement, $\nu_-(\Theta_1^{-1}\Theta)$ is the  number of negative squares of the kernel
\[
  \frac{X(J,r)-\Theta_1(r(z))^{-1}\Theta(r(z))X(J,r)\Theta(r(w))^*\Theta_1(r(z))^{-*}}{-i(r(z)-\overline{r(w)})}
\]
(note that the numerator involves only $X(J,r)$ and not its inverse).
\begin{theorem}
  In the notation of the previous theorem, let $\mathfrak P_1$ be a subspace of $\mathfrak P$ which is $R_\alpha^{(r)}$-invariant, and non-degenerate. Then $\mathfrak P_1$ has a reproducing kernel of the
  form
  \[
( Z_r(z)\otimes I_p)\frac{X(J,r)^{-1}-\Theta_1(r(z))X(J,r)\Theta_1(r(w))^*}{-i(r(z)-\overline{r(w)})}(Z_r(w)\otimes I_p)^*
\]
and we have
\[
\nu_-(\Theta)=\nu_-(\Theta_1)+\nu_-(\Theta_1^{-1}\Theta)
\]
and the direct and orthogonal decomposition
\begin{equation}
\mathfrak P(\Theta)=\mathfrak P(\Theta_1)\oplus (Z_r\otimes I_p)\Theta_1(r)\mathfrak P(\Theta_1^{-1}\Theta).
\end{equation}
\end{theorem}

\begin{proof} Thanks to Corollary \ref{corocoro1} the claim reduces to the decomposition
  \[
    \begin{split}
      \frac{X(J,r)^{-1}-\Theta(r(z))X(J,r)\Theta(r(w))^*}{-i(r(z)-\overline{r(w)})}&=\frac{X(J,r)^{-1}-\Theta_1(r(z))X(J,r)\Theta_1(r(w))^*}{-i(r(z)-\overline{r(w)})}+\\
      &\hspace{-6cm}+\Theta_1(r(z))\frac{X(J,r)-\Theta_1(r(z))^{-1}\Theta(r(z))X(J,r)\Theta(r(w))^*\Theta_1(r(z))^{-*}}{-i(r(z)-\overline{r(w)})}\Theta_1(r(z))^*,
    \end{split}
    \]
and the result follows then from classical results in reproducing kernel Pontryagin spaces.
\end{proof}

We now turn to the circle case. The argument is similar to the one in the previous theorem, but we need to use Corollary \ref{c123} at some point $w_0$ of modulus $1$ not a pole of $r$. Recall that
$A(r,w_0)$ was introduced in Definition \ref{def567}.

\begin{theorem} Let $r$ be a rational function satisfying the hypothesis of Theorem \ref{234432}, let $w_0$ be a point on the unit circle not a pole of $r$.
    Let $\mathfrak P$ be a reproducing kernel Pontryagin space of $\mathbb C^p$-valued functions analytic in $r^{-1}(A(w_0,r))$. Then, for each $\alpha,\beta\in
    \omega(r)\cap A(w_0,r)$ it holds that
    %Let $f$ be analytic in $\Omega$ and assume that, $f(\Omega)\cap\mathbb R\not=\emptyset$ and that for $\alpha,\beta\in f(\Omega)\cap{\rm ran}\,r$ such that the
\begin{equation}
  \left[ f,g\right]_{\mathfrak P} +\overline{\beta}\left[ f, R_\beta^{(r)} g\right]_{\mathfrak P} +\alpha\left[R_\alpha^{(r)}f,g\right]_{\mathfrak P} -(1-\alpha\overline{\beta})
  \left[ R^{(r)}_\alpha f,R^{(r)}_\beta g\right]_{\mathfrak P}
    =(\alpha-\overline{\beta})\sum_{i,j=1}^N
  \frac{g(v_j)^*}{\overline{r^\prime(v_j)}}J\frac{f(u_i)}{r^\prime(u_i)}\frac{1}{u_i-\overline{v_j}}
\end{equation}
(with $r(u_i)=\alpha$ and $r(v_j)=\beta$, $i,j=1,\ldots, N$)
  if and only if the reproducing kernel of $\mathfrak P$ is of the form
  \begin{equation}
 ( Z_r(z)\otimes I_p)\frac{X(J,r)^{-1}-\Theta(r(z))X(J,r)\Theta(r(w))^*}{1-r(z)\overline{r(w)}}Y^{-*}(Z_r(w)\otimes I_s)^*,
\end{equation}
where $\Theta$ is a $\mathbb C^{Np\times Np}$-valued function analytic in a neighborhood of the origin.
\label{tm72}
  \end{theorem}

  \begin{proof}
The functions $F$ are now analytic across the unit circle at $w_0$ and one can use the Pontryagin space version of de Branges' theorem for the circle case; see \cite[Theorem 6.10]{ad3} for the latter.
   \end{proof}

%\begin{thebibliography}{99}
%\bibitem{CAPB_2} D. Alpay,
%  \emph{An advanced complex analysis problem book. Topological Vector Spaces.
%  Functional Analysis and Hilbert spaces of analytic functions}. Birk\"auser/Springer
%Basel AG, Basel 2015.
% \end{thebibliography}

\section{$\mathfrak P(S)$ spaces}
\setcounter{equation}{0}

It is interesting to note that the Fock space restricted to the unit disk
is a special case of a $\mathfrak P(S)$ space; see \cite{MR4032209}.

\begin{theorem}
  Let $r$ be a {\rm real} rational function satisfying the hypothesis of Theorem \ref{234432}.
  Let $\mathfrak P$ be a reproducing kernel Pontryagin space of $\mathbb C^p$-valued functions analytic in $r^{-1}(A(r))\setminus\mathbb R$, and let $\alpha\in\Omega(r)$.
  Assume that the inequality
\begin{equation}
\label{inequality34}
  \left[ R^{(r)}_\alpha f,R^{(r)}_\alpha f\right]_{\mathfrak P}\le\left[ (I+\alpha R^{(r)}_\alpha) f,(I+\alpha R^{(r)}_\alpha) f\right]_{\mathfrak P}
  -(\alpha-\overline{\alpha})\sum_{i,j=1}^N
  \frac{f(u_j)^*}{\overline{r^\prime(u_j)}}J\frac{f(u_i)}{r^\prime(u_i)}\frac{1}{u_i-\overline{u_j}}
\end{equation}
holds for some preassigned signature matrix $J\in\mathbb R^{p\times p}$.
Then there exist a Pontryagin space $\mathfrak C$ such that $\nu_-(\mathfrak C)=\nu_-(X)$,
and a $\mathbf L(\mathfrak C,\mathbb C^{Np}_{J_0})$-valued function analytic in a neighborhood of $\alpha$ such that the reproducing kernel of $\mathfrak P$
is of the form
\begin{equation}
  K(z,w)=(Z_r(z)\otimes I_p)
\frac{X(J,r)^{-1}-S(r(z))S(r(w))^*}{1-r(z)\overline{r(w)}}
(Z_r(w)\otimes I_p)^*.
\end{equation}

\end{theorem}

\begin{proof} The strategy is the same as for Theorem \ref{23443211}. We write elements of $\mathfrak P$ in the form $f=(Z_r\otimes I_p)F(r))$, where $F$ is uniquely
  determined and analytic in $A(r)$, and with induced inner product \eqref{inner1}. Inequality \eqref{inequality34} becomes
  \begin{equation}
    \label{inequality34321}
[R_\alpha F,R_\alpha F]\le[(I+\alpha R_\alpha) F,(I+\alpha R_\alpha )F]-F(\alpha)^*XF(\alpha)
\end{equation}
(see \cite[]{}, \cite[(6.9)]{ad3}).
As above, write $X=Y^*J_0Y$. The space of functions of the form $YF$ with norm the norm of $F$ has reproducing kernel of the form
\[
\frac{J_0-S(r(z))S(r(w)^*}{1-r(z)\overline{r(w)}}
\]
where $\mathfrak C$ is a Pontryagin space with negative index  $\nu_-(J_0)$ and where $S$ is a $\mathbf L(\mathfrak C,\mathbb C_{J_0}^{Np})$-valued function analytic in a
neighborhood of $\alpha$. This is \cite[Theorem 3.1.2, p. 85]{adrs} for the case $\alpha=0$ and \cite[Theorem 3.4, p. 32]{MR3672946} for the general case.
\end{proof}

Classical $\mathcal H(s)$ spaces are generalizations of the orthogonal complements of Beurling-Lax invariant subspaces. This suggests that the counterpart of the
Hardy space in the current setting is the reproducing kernel space with reproducing kernel
\begin{equation}
(Z_r(z)\otimes I_p)
\frac{X(J,r)^{-1}}{1-r(z)\overline{r(w)}}
(Z_r(w)\otimes I_p)^*.
\end{equation}

Let us set $p=1$ for simplicity. The expression
\[
Z_r(z)X(J,r)^{-1}Z_r(w)
\]
is a reproducing kernel formula.

\begin{remark}{\rm
The result \cite[Theorem 3.4, p. 32]{MR3672946} will also allow to get the half-line counterpart of the above theorem, using the inequality
(see \cite[Theorem 6.7 (6.10)]{ad3})
\[
\frac{\alpha-\overline{\alpha}}{2\pi i}[R_\alpha F,R_\alpha F]\le \frac{[R_\alpha F, F]-[F,R_\alpha F]}{2\pi i}-F(\alpha)^*F(\alpha)
\]
rather than \eqref {inequality34321}, and the relation
\[
\frac{\alpha-\overline{\alpha}}{2\pi i}[R^{(r)}_\alpha F,R^{(r)}_\alpha F]\le \frac{[R^{(r)}_\alpha F, F]-[F,R^{(r)}_\alpha F]}{2\pi i}-(\alpha-\overline{\alpha})\sum_{i,j=1}^N
  \frac{f(u_j)^*}{\overline{r^\prime(u_j)}}J\frac{f(u_i)}{r^\prime(u_i)}\frac{1}{u_i-\overline{u_j}}
\]
rather than \eqref{inequality34}.}
\end{remark}

\section{$\mathfrak P(E_+,E_-)$ spaces}
\setcounter{equation}{0}

We now present two structure theorems. The proof of the first one is a matrix-valued version of \cite[Th\'eor\`eme 23, p. 59]{dbbook}. The proof of the second theorem is an
adaptation to the circle case of de Branges' proof, and appears in \cite{a-seville} and
\cite[Theorem 6.1, p. 173]{abl3}. Both proofs are repeated here for completeness.

\begin{theorem} (The line case)
  Let $\mathfrak P$ be a reproducing kernel Pontryagin space of $\mathbb C^p$ functions analytic in an open set $\Omega$, symmetric with respect to the real line and intersecting it,  and assume that
there is a point  $\alpha\not=0\in\Omega$ such that both $K(\alpha,\alpha)$ and $K(\overline{\alpha},\overline{\alpha})$ are invertible.
Assume that the following condition holds for every $F,G\in \mathfrak P$: if $F(v)=G(v)=0$ for $v\not=\overline{v}\in\Omega$ then the functions $z\mapsto \frac{z-\overline{v}}{z-v}F(z)$
and $z\mapsto \frac{z-\overline{v}}{z-v}G(z)$ belongs to $\mathfrak P$ and
\begin{equation}
  \label{fgfg}
\left[F(z),G(z)\right]=\left[\frac{z-\overline{v}}{z-v}F(z), \frac{z-\overline{v}}{z-v}G(z)\right].
\end{equation}
Then, there exist $\mathbb C^{p\times p}$-valued functions $E_+$ and $E_-$ and a signature matrix $J\in\mathbb c^{p\times p}$ such that the reproducing kernel of $\mathfrak P$ is of the form
\begin{equation}
  \label{eee}
  K_{\mathfrak P}(z,w)=\frac{E_+(z)JE_+(w)^*-E_-(z)JE_-(w)^*}{-i(z-\overline{w})},
\end{equation}
\label{t123}
where $J$ is a signature matrix with same signature as $K(\alpha,\alpha)$.
\end{theorem}

\begin{proof}
  Without loss of generality, and by maybe replacing $\alpha$ by its conjugate, we will assume ${\rm Im}\,\alpha>0$. Let $c\in\mathbb C^p$. The function
    \begin{equation}
    \label{dbcircle}
z\mapsto\left(K(z,w)-K(z,\alpha)K(\alpha,\alpha)^{-1}K(\alpha,w)\right)c
\end{equation}
belongs to $\mathfrak P
$ and vanishes at the point $\alpha$. So the function
  \[
z\mapsto\frac{z-\overline{\alpha}}{z-\alpha}\left(K(z,w)-K(z,\alpha)K(\alpha,\alpha)^{-1}K(\alpha,w)\right)c
\]
belongs to $\mathfrak P$.
Let $F\in\mathfrak P$  which vanishes at $\overline{\alpha}$. Using \eqref{fgfg} with $v=\overline{\alpha}$ we can write
\[
\begin{split}
  \left[F(z),\frac{z-\overline{\alpha}}{z-\alpha}\left(K(z,w)-K(z,\alpha)K(\alpha,\alpha)^{-1}K(\alpha,w)\right)c\right]&=\\
  &\hspace{-8cm}=
\left[F(z)\frac{z-\alpha}{z-\overline{\alpha}},\frac{z-\alpha}{z-\overline{\alpha}}\frac{z-\overline{\alpha}}{z-\alpha}\left(K(z,w)-K(z,\alpha)K(\alpha,\alpha)^{-1}K(\alpha,w)\right)c\right]\\
&\hspace{-8cm}= \left[F(z)\frac{z-\alpha}{z-\overline{\alpha}},\left(K(z,w)-K(z,\alpha)K(\alpha,\alpha)^{-1}K(\alpha,w)\right)c\right]\\
&\hspace{-8cm}= \left[F(z)\frac{z-\alpha}{z-\overline{\alpha}},\left(K(z,w)-K(z,\alpha)K(\alpha,\alpha)^{-1}K(\alpha,w)\right)c\right]\\
&\hspace{-8cm}=c^*F(w)\frac{w-\alpha}{w-\overline{\alpha}}\\
&\hspace{-8cm}=\left[F(z),\left(K(z,w)-K(z,\overline{\alpha})K(\overline{\alpha},\overline{\alpha})^{-1}K(\overline{\alpha},w)\right)c
  \frac{\overline{w}-\overline{\alpha}}{\overline{w}-\alpha}\right].
\end{split}
\]

This chain of equalities holds for all $F$ vanishing at $\overline{\alpha}$ and hence we get
\[
  \begin{split}
  \frac{z-\overline{\alpha}}{z-\alpha}\left(K(z,w)-K(z,\alpha)K(\alpha,\alpha)^{-1}K(\alpha,w)\right)&=\\
  &\hspace{-4cm}=
\left(K(z,w)-K(z,\overline{\alpha})K(\overline{\alpha},\overline{\alpha})^{-1}K(\overline{\alpha},w)\right)\frac{\overline{w}-\overline{\alpha}}{\overline{w}-\alpha}.
\end{split}
\]
We can thus write:
\[
  \begin{split}
  \left(
    \frac{z-\overline{\alpha}}{z-\alpha}
    -\frac{\overline{w}-\overline{\alpha}}{\overline{w}-\alpha}\right)
  K(z,w)&=\\
  &\hspace{-4cm}=
  \frac{z-\overline{\alpha}}{z-\alpha}K(z,\alpha)K(\alpha,\alpha)^{-1}K(\alpha,w)-\frac{\overline{w}-\overline{\alpha}}{\overline{w}-\alpha}
  K(z,\overline{\alpha})K(\overline{\alpha},\overline{\alpha})^{-1}K(\overline{\alpha},w).
  \end{split}
\]
Therefore:
\begin{equation}
  K(z,w)=\frac{F_+(z)K(\alpha,\alpha)^{-1}F_+(w)^*-F_-(z)K(\overline{\A},\overline{\A})^{-1}F_-(w)^*}{(-i(\alpha-\overline{\alpha}))(-i(z-\overline{w}))},
\end{equation}

with
\begin{eqnarray}
  F_+(z)&=&\frac{1}{\sqrt{-i(\alpha-\overline{\alpha})} }           (z-\overline{\alpha})K(z,\A)\\
  F_-(z)&=&\frac{1}{\sqrt{-i(\alpha-\overline{\alpha})}}(z-{\alpha})K(z,\overline{\A}).
\end{eqnarray}
By analyticity, $\det K(z,\A)\not\equiv0$ and so taking a point on the unit circle where $\det  K(z,\A)\not=0$ and $\det K(z,\overline{\A})\not=0$ we see that $K(\alpha,\alpha)$ and
$K(\overline{\A},\overline{\A})$ have the same signature, and we can write with a common signature matrix $J$
\[
  K(\alpha,\alpha)^{-1}=MJM^*\quad{\rm and}\quad K(1/\overline{\A},1/\overline{\A})^{-1}=NJN^*.
\]
The result follows with $E_+=F_+M$ and $E_-=F_-N$.

\end{proof}
\begin{theorem} (The circle case) Let $\mathfrak P$ be a reproducing kernel Pontryagin space of $\mathbb C^p$ functions analytic in an open set $\Omega$, symmetric with respect to the unit
  circle and intersecting it, and assume that there is a point  $\alpha\not=0\in\Omega$ such that both $K(\alpha,\alpha)$ and $K(1/\overline{\alpha},1/\overline{\alpha})$ are invertible.
Assume that the following condition holds for every $F,G\in \mathfrak P$: if $F(v)=G(v)=0$ for $v\in\Omega\setminus\mathbb T$ then the functions $z\mapsto \frac{1-z\overline{v}}{z-v}F(z)$
 $z\mapsto \frac{1-z\overline{v}}{z-v}G(z)$ belongs to
belongs to
$\mathfrak P$ and
\begin{equation}
  \label{fgfg1}
\left[F(z),G(z)\right]=[\frac{1-z\overline{v}}{z-v}F(z), \frac{1-z\overline{v}}{z-v}G(z)].
\end{equation}
Then, there exist $\mathbb C^{p\times p}$-valued functions $E_+$ and $E_-$ and a signature matrix $J\in\mathbb C^{p\times p}$ such that the reproducing kernel of $\mathfrak P$ is of the form
\begin{equation}
  \label{eeee}
K_{\mathfrak P}(z,w)=\frac{E_+(z)JE_+(w)^*-E_-(z)JE_-(w)^*}{1-z\overline{w}},
\end{equation}
\label{t1234}
where $J$ is a signature matrix with same signature as $K(\alpha,\alpha)$.
\end{theorem}

\begin{proof}
Let $c\in\mathbb C^p$. The function \eqref{dbcircle}
and vanishes at the point $\alpha$. So the function
  \[
z\mapsto\frac{1-z\overline{\alpha}}{z-\alpha}\left(K(z,w)-K(z,\alpha)K(\alpha,\alpha)^{-1}K(\alpha,w)\right)c\in\mathfrak P.
\]
Let $F\in\mathfrak P$  which vanishes at $1/\overline{\alpha}$. Using \eqref{fgfg1} with $v=1/\overline{\alpha}$ we can write
\[
\begin{split}
  \left[F(z),\frac{1-z\overline{\alpha}}{z-\alpha}\left(K(z,w)-K(z,\alpha)K(\alpha,\alpha)^{-1}K(\alpha,w)\right)c\right]&=\\
  &\hspace{-8cm}=
\left[F(z)\frac{1-z/\alpha}{z-1/\overline{\alpha}},\frac{1-z/\alpha}{z-1/\overline{\alpha}}\frac{1-z\overline{\alpha}}{z-\alpha}\left(K(z,w)-K(z,\alpha)K(\alpha,\alpha)^{-1}K(\alpha,w)\right)c\right]\\
&\hspace{-8cm}= \left[F(z)\frac{1-z/\alpha}{z-1/\overline{\alpha}},\frac{\overline{\alpha}}{\alpha}\left(K(z,w)-K(z,\alpha)K(\alpha,\alpha)^{-1}K(\alpha,w)\right)c\right]\\
&\hspace{-8cm}= \left[F(z)\frac{z-\alpha}{1-z\overline{\alpha}},\left(K(z,w)-K(z,\alpha)K(\alpha,\alpha)^{-1}K(\alpha,w)\right)c\right]\\
&\hspace{-8cm}=c^*F(w)\frac{w-\alpha}{1-w\overline{\alpha}}\\
&\hspace{-8cm}=\left[F(z),\left(K(z,w)-K(z,1/\overline{\alpha})K(1/\overline{\alpha},1/\overline{\alpha})^{-1}K(1/\overline{\alpha},w)\right)c
  \frac{\overline{w}-\overline{\alpha}}{1-\overline{w}\alpha}\right].
\end{split}
\]
This chain of equalities holds for all $F$ vanishing at $1/\overline{\alpha}$ and hence we get
\[
  \begin{split}
  \frac{1-z\overline{\alpha}}{z-\alpha}\left(K(z,w)-K(z,\alpha)K(\alpha,\alpha)^{-1}K(\alpha,w)\right)&=\\
  &\hspace{-4cm}=
\left(K(z,w)-K(z,1/\overline{\alpha})K(1/\overline{\alpha},1/\overline{\alpha})^{-1}K(1/\overline{\alpha},w)\right)\frac{\overline{w}-\overline{\alpha}}{1-\overline{w}\alpha}.
\end{split}
\]
Hence
\[
  \begin{split}
  \left(
    \frac{1-z\overline{\alpha}}{z-\alpha}
    -\frac{\overline{w}-\overline{\alpha}}{1-\overline{w}\alpha}\right)
  K(z,w)&=\\
  &\hspace{-4cm}=
  \frac{1-z\overline{\alpha}}{z-\alpha}K(z,\alpha)K(\alpha,\alpha)^{-1}K(\alpha,w)-\frac{\overline{w}-\overline{\alpha}}{1-\overline{w}\alpha}
  K(z,1/\overline{\alpha})K(1/\overline{\alpha},1/\overline{\alpha})^{-1}K(1/\overline{\alpha},w),
  \end{split}
\]
so that
\begin{equation}
  K(z,w)=\frac{F_+(z)K(\alpha,\alpha)^{-1}F_+(w)^*-F_-(z)K(1/\overline{\A},1/\overline{\A})^{-1}F_-(w)^*}{(1-|\A|^2)(1-z\overline{w})},
\end{equation}

with
\begin{eqnarray}
  F_+(z)&=&\frac{1}{\sqrt{1-|\A|^2}}
            (1-z\overline{\alpha})K(z,\A)\\
  F_-(z)&=&\frac{1}{\sqrt{1-|\A|^2}}(z-{\alpha})K(z,1/\overline{\A}).
\end{eqnarray}
By analyticity, $\det K(z,\A)\not\equiv0$ and so taking a point on the unit circle where $\det  K(z,\A)\not=0$ and $\det K(z,1/\overline{\A})\not=0$ we see that $K(\alpha,\alpha)$ and
$K(1/\overline{\A},1/\overline{\A})$ have the same signature, and we can write with a common signature matrix $J$
\[
  K(\alpha,\alpha)^{-1}=MJM^*\quad{\rm and}\quad K(1/\overline{\A},1/\overline{\A})^{-1}=NJN^*.
\]
The result follows with $E_+=F_+M$ and $E_-=F_-N$.
\end{proof}

In the case of a Hilbert space we have $J=I_p$ and the following proposition holds:

\begin{proposition}
$K(\alpha,\alpha)$ is invertible if and only if linear span of the functions $f(\alpha)$ when $f$ runs through $\Omega$ is equal to $\mathbb C^p$.
\end{proposition}

\begin{proof}
  Assume that there is $c\in\mathbb C^p$ such that $K(\alpha,\alpha)c=0$  Then, the Cauchy-Schwarz inequality implies that
  \[
    \begin{split}
      |c^*f(\alpha)|^2 &=|\langle f,K(\cdot,\alpha)\rangle|^2\\
      &\le  \|f\|^2\cdot(c^*K(\alpha,\alpha)c)\\
      &=0
    \end{split}
    \]
    and so $c$ is orthogonal to the said linear span.\smallskip

    Conversely, if there is such a $c$, the map $f\mapsto c^*f(\alpha)$ is identically equal to $0$, and hence its Riesz representation corresponds to the zero vector, i.e.
    $K(\cdot,\alpha)c\equiv 0$, and so in particular $K(\alpha,\alpha)c=0$.
\end{proof}

In the case of a Pontryagin space, the second part of the proof is still valid, since Riesz representation theorem holds in Pontryagin space, and we have:

\begin{corollary}
Assume that the linear span is not dense. Then $K(\alpha,\alpha)$ is not invertible.
\end{corollary}

Next theorem further generalizes results in Section \ref{Sec6}:

\begin{theorem}
  Let $r$ be a rational function satisfying the hypothesis of Theorem \ref{234432}, and let $A(r)$ be defined by \eqref{domain123}.
  Let $\mathfrak P$ be a reproducing kernel Pontryagin space of $\mathbb C^p$-valued functions analytic in $r^{-1}(A(r))$ and having
  full rank, in the sense that there exist  $\alpha\not=\overline{\alpha}\in\Omega_0$
    such that the matrices $K(\alpha,\alpha)$ and $K(\overline{\alpha},\overline{\alpha})$ are invertible. Assume that for $f,g\in \mathfrak P$ and $v\in \Omega$,
\[
  f(w_n)=g(w_n)=0,\quad n=1,\ldots, N\,\,\Longrightarrow f(z)\frac{r(z)+\overline{v}}{r(z)-v}\,\, and\,\,\,
  g(z)\frac{r(z)+\overline{v}}{r(z)-v}  \,\, \in\mathfrak P,
\]
where $w_1,\ldots, w_n$ are the roots, assumed pairwise different, of $r(w)=v$, and
\begin{equation}
  \label{fghj}
\left[f(z),g(z)\right]_{\mathfrak P}=\left[f(z)\frac{r(z)+\overline{\alpha}}{r(z)-\alpha},g(z)\frac{r(z)+\overline{\alpha}}{r(z)-\alpha}\right]_{\mathfrak P}.
\end{equation}
Then the reproducing kernel of $\mathfrak P$ is of the form
\[
  K(z,w)=(Z_r(z)\otimes I_p)  \frac{E_+(r(z))JE_+(r(w))^*-E_-(r(z))JE_-(r(w))^*}{-i(r(z)-\overline{r(w)})}(Z_r(w)\otimes I_p)^*,
\]
where $J$ is a signature matrix with same signature as $K(\alpha,\alpha)$.
\end{theorem}

\begin{proof} Using the vector-valued version of Lemma \ref{le123211} we write $f(z)=Z_r(z)F(r(z))$, where $F$ is $\mathbb C^{Np}$-valued and analytic in $\Omega_0$.
  We consider the space of functions
  \begin{equation}
    \label{mmmm}
\mathfrak M=\left\{F(z)\,:\, Z_r(z)F(r(z))\in\mathfrak H,\,\,{\rm with\,\, inner\,\, product}\,\,\ [F,G]=[f,g]_{\mathfrak P}\right\}.
  \end{equation}
By Corollary \ref{corocoro}  the condition $f(w_n)=0$, $n=1,\ldots, N$ implies that $F(\alpha)=0$.\\
Equation \eqref{fghj} implies that
\[
\left[F(z),G(z)\right]=\left[F(z)\frac{z+\overline{v}}{z-v},G(z)\frac{z+\overline{v}}{z-v}\right].
\]
By Theorem \ref{t123} the reproducing kernel of the form \eqref{eee} and hence the result.

\end{proof}

In the case of the circle we have the following result; the proof is similar and will be omitted.

\begin{theorem}
    Let $r$ be a rational function satisfying the hypothesis of Theorem \ref{234432}, and let $A(r)$ be defined by \eqref{domain123}.
  Let $\mathfrak P$ be a reproducing kernel Pontryagin space of $\mathbb C^p$-valued functions analytic in $r^{-1}(A(r))$ and having
  full rank, in the sense that there exist  $\alpha\not=\overline{\alpha}\in\Omega_0$
    such that the matrices $K(\alpha,\alpha)$ and $K(1/\overline{\alpha},1/\overline{\alpha})$ are invertible. Assume that for $f,g\in \mathfrak P$ and $v\in \Omega$
\[
  f(w_n)=g(w_n)=0,\quad n=1,\ldots, N\,\,\Longrightarrow f(z)\frac{1-r(z)\overline{\alpha}}{r(z)-\overline{\alpha}}\,\, and\,\,\,
  g(z)\frac{1-r(z)\overline{\alpha}}{r(z)-\alpha}  \,\, \in\mathfrak P,
\]
where $w_1,\ldots, w_n$ are the roots, assumed pairwise different, of $r(w)=v$, and
\begin{equation}
  \label{fghj}
\left[f(z),g(z)\right]_{\mathfrak P}=\left[f(z)\frac{1-r(z)\overline{v}}{r(z)-v},g(z)\frac{1-r(z)\overline{v}}{r(z)-v}\right]_{\mathfrak P}.
\end{equation}
Then the reproducing kernel of $\mathfrak P$ is of the form
\[
  K(z,w)=(Z_r(z)\otimes I_p)  \frac{E_+(r(z))JE_+(r(w))^*-E_-(r(z))JE_-(r(w))^*}{-i(r(z)-\overline{r(w)})}(Z_r(w)\otimes I_p)^*,
\]
where $J$ is a signature matrix with same signature as $K(\alpha,\alpha)$.
\label{tm102}
\end{theorem}

The previous theorems are of special interest in the Hilbert space setting, when the space $\mathfrak M$ (see \eqref{mmmm}) consists of entire functions.
\section{$\mathfrak L({\mathsf N})$ spaces}
\setcounter{equation}{0}

We first recall:

\begin{definition}
  The function ${\mathsf N}$ analytic in the open set $\Omega\subset\mathbb C$, symmetric with respect to the real line is called a generalized Nevanlinna function if the kernel
  \[
    K_{\mathsf N}(z,w)=\frac{{\mathsf N}(z)-{\mathsf N}(w)^*}{z-\overline{w}}
  \]
has a finite number of negative squares in $\Omega$.
\end{definition}

The following result appears, for the scalar and positive setting, in \cite[Th\'eor\`eme 6, p. 13]{dbbook}; a first Pontryagin space version, with more restrictive conditions and different proof,
appears in \cite{alpaythesis}. Here we are treating the case of negative squares,
which corresponds to generalized Nevanlinna functions. The latter class of functions has been introduced and thoroughly studied by Krein and Langer;
see e.g. \cite{kl1, MR47:7504}. These functions are closely related to the positive real functions and generalized positive real functions from linear system theory; see
e.g. \cite{DDGK}.

\begin{theorem}
  \label{rty789}
  Let $\mathfrak P$ be a reproducing kernel Pontryagin space of $\mathbb C^p$-valued functions analytic in a set $\Omega$ symmetric with respect to the real line and assume that
  $\mathfrak P$ is resolvent invariant. Assume moreover that
  \begin{equation}
    \label{lphi}
[F(z), (R_\beta G)(z)]_{\mathfrak P}-[(R_\alpha F)(z), F(z)]_{\mathfrak P}+(\alpha-\overline{\beta})[(R_\alpha F)(z), (R_\beta G)(z)]_{\mathfrak P}=0
\end{equation}
for every $\alpha,\beta\in\Omega$.
Then there is a $\mathbb C^{p\times p}$-valued function ${\mathsf N}(z)$ analytic in $\Omega$ and such that the reproducing kernel of $\mathfrak P$ is
\begin{equation}
K_{\mathsf N}(z,w)=\frac{{\mathsf N}(z)-{\mathsf N}(w)^*}{z-\overline{w}},\quad z,w\in\Omega.
\end{equation}
\end{theorem}

  \begin{proof}
    The core of the proof is as in \cite[p. 13]{dbbook}, and we adapt the arguments to the case of negative squares:\\

    STEP 1: {\sl The resolvent operators are bounded in $\mathfrak P$:}\smallskip

    Indeed, they are closed in view of the reproducing kernel property, and everywhere defined. Hence they are bounded since we are in a Pontryagin space.\\

    STEP 2: {\sl The reproducing kernel $K(z,w)$ satisfies:}
    \begin{equation}
      \label{db84}
  (\overline{w}-z)K(z,w)=-(z-\overline{\alpha})K(z,\alpha)+(\overline{w}-\overline{\alpha})K(\overline{\alpha},w),\quad \alpha,z,w\in\Omega.
\end{equation}

As in \cite[p. 13]{dbbook} this follows from setting $G(z)=K(z,w)d$ and $\alpha=\overline{\beta}$ in \eqref{lphi}. We then get for any $F\in\mathfrak P$
\[
  \begin{split}
    \left[F(z),\frac{K(z,w)-K(\overline{\alpha},w)}{z-\overline{\alpha}}d\right]&=\left[\frac{F(z)-F({\alpha})}{z-{\alpha}},K(z,w)d\right]\\
    &=d^*\frac{F(w)-F(\alpha)}{w-{\alpha}}\\
    &=\left[F(z),\frac{K(z,w)-K(z,\alpha)}{\overline{w}-\overline{\alpha}}d\right]
  \end{split}
  \]
and so
\[
  \frac{K(z,w)-K(\overline{\alpha},w)}{z-\overline{\alpha}}=\frac{K(z,w)-K(z,\alpha)}{\overline{w}-\overline{\alpha}},
\]
from which \eqref{db84} follows.\\

It follows from the previous step that
\begin{equation}
  K(z,w)=\frac{{\mathsf N}_1(z)-{\mathsf N}_2(w)^*}{z-\overline{w}}
\end{equation}
with ${\mathsf N}_1(z)=(z-\overline{\alpha})K(z,\alpha)$ and ${\mathsf N}_2(w)=(w-\alpha)K(w,\overline{\alpha})$.\\

STEP 3: {\sl It holds that ${\mathsf N}_1(z)-{\mathsf N}_2(w)=(\alpha-\overline{\alpha})K(\alpha,\alpha)$.}\smallskip

Indeed, from the equality $K(z,w)=K(w,z)^*$ we have
\[
\frac{{\mathsf N}_1(z)-{\mathsf N}_2(w)^*}{z-\overline{w}}=\left(\frac{{\mathsf N}_1(w)-{\mathsf N}_2(z)^*}{w-\overline{z}}\right)^*,
\]
from which we get
\[
  {\mathsf N}_1(z)-{\mathsf N}_2(w)^*={\mathsf N}_2(z)-{\mathsf N}_1(w)^*
\]
and hence the result.\\

We conclude by taking (see \cite[p. 14]{dbbook}),
\[
  {\mathsf N}(z)=\frac{(\alpha-\overline{\alpha})}{2}K(\alpha,\alpha)+(z-\overline{\alpha})K(z,\alpha).
\]
Then, $K(z,w)=\frac{{\mathsf N}(z)-{\mathsf N}(w)^*}{z-\overline{w}}$.
   \end{proof}

\begin{theorem}
  \label{rty78910}
 Let $r$ be a rational function satisfying the hypothesis of Theorem \ref{234432}, and let $A(r)$ be defined by \eqref{domain123}.
  Let $\mathfrak P$ be a reproducing kernel Pontryagin space of $\mathbb C^p$-valued functions analytic in $r^{-1}(A(r))$, and assume that
  $\mathfrak P$ is $R_\alpha^{(r)}$ resolvent invariant. Assume moreover that
  \begin{equation}
    \label{lphi}
[F(z), (R_\beta^{(r)} G)(z)]_{\mathfrak P}-[(R_\alpha^{(r)} F)(z), F(z)]_{\mathfrak P}+(\alpha-\overline{\beta})[(R_\alpha^{(r)} F)(z), (R_\beta^{(r)} G)(z)]_{\mathfrak P}=0.
\end{equation}
Then there is a $\mathbb C^{p\times p}$-valued function ${\mathsf N}(z)$ analytic in $\Omega$ and such that the reproducing kernel of $\mathfrak P$ is
\begin{equation}
K(z,w)=Z_r(z)\frac{{\mathsf N}(r(z))-{\mathsf N}(r(w))^*}{r(z)-\overline{r(w)}}Z_r(w)^*,\quad z,w\in r^{-1}(A(r)).
\end{equation}
\end{theorem}

  \begin{proof}
    As in earlier proofs in the paper, we using \eqref{formulacoro} and $f=Z_rF(r)$ and $g=Z_rG(r)$; we consider the space of functions
    corresponding space of functions $F$ with the inner product
    \[
[F,G]=[f,g].
\]
We obtain by Theorem \ref{rty789} a space $\mathfrak L({\mathsf N})$ and the result follows.
   \end{proof}

  When ${\mathsf N}(z)\equiv i/2$ in the open upper half-plane, we get the counterpart of the Hardy space
  \[
\frac{Z_r(z)Z_r(w)^*}{-i(r(z)-\overline{r(w)})}=\sum_{n=1}^N \frac{e_n(z)\overline{e_n(w)}}{-i(r(z)-\overline{r(w)})},
  \]
  which is positive definite on the set of $w$ such that  ${\rm Im}\, r(w)\ge 0$ and with associated reproducing kernel Hilbert space
  \[
f(z)=\sum_{n=1}^N e_n(z)f_n(r(z)),\quad f_1,\ldots, f_N\in\mathbf H^2(\mathbb C_+).
  \]
  The new Hardy space reduces to the classical one when
  \[
    \sum_{n=1}^Ne_n(z)\overline{e_n(w)}=\frac{-i(r(z)-\overline{r(w)})}{-i(z-\overline{w})}.
   \]
This will happen if and only if $r$ satisfies the symmetry $r(z)=\overline{r(\overline{z})}$. See \cite{ajlm2} for the corresponding Cuntz type decomposition.\\

%%%%%%%%%%%%%%%%%%%%%%%%%%%%%

%%%%%%%%%%%%%%%%%%%%%%%%%%%%%
%%%%%%%%%%%%%%%%%%%%%%%%%%%%%

%\section{Example: The case of a Blaschke product}
%\setcounter{equation}{0}

{\bf Acknowledgements:} It is a pleasure to thank Prof. Dan Volok, Kansas State University, for help with Theorem \ref{234432}.

\end{document}